\documentclass[11pt]{amsart}
\usepackage{amsmath, amscd, amssymb}
\usepackage{latexsym, amssymb, amsmath, amscd, amsfonts}
\usepackage[mathscr]{eucal}
\usepackage{mathrsfs}
\usepackage{color}
\usepackage{concrete,euler,textcomp}
\usepackage{young, youngtab} 
\usepackage[T1]{fontenc}

\numberwithin{equation}{subsection}
\newtheorem{theorem}{Theorem}[subsection]
\newtheorem{corollary}[theorem]{Corollary}
\newtheorem{definition}[theorem]{Definition}
\newtheorem{example}[theorem]{Example}
\newtheorem{lemma}[theorem]{Lemma}
\newtheorem{proposition}[theorem]{Proposition}
\newtheorem{remark}[theorem]{Remark}
\newtheorem{notation}[theorem]{Notation}
\begin{document}

\title[Stable Range Branching Algebras]{Distributive Lattices, Affine Semigroups, and\\
Branching Rules of the Classical Groups}
\author{Sangjib Kim}
\email{skim@maths.uq.edu.au}
\address{School of Mathematics and Physics, The University of Queensland, St Lucia, QLD 4072, Australia}

\begin{abstract}
We study algebras encoding stable range branching
rules for the pairs of complex classical groups of the same type in the
context of toric degenerations of spherical varieties. By lifting affine semigroup
algebras constructed from combinatorial data of branching multiplicities, we
obtain algebras having highest weight vectors in multiplicity spaces as
their standard monomial type bases. In particular, we identify a family of 
distributive lattices and their associated Hibi algebras which can uniformly 
describe the stable range branching algebras for all the pairs we consider.
\end{abstract}

\maketitle

\section{Introduction}

\subsection{}
Let us consider a pair of complex algebraic groups $G$ and $H$ with
embedding $H \subset G$ and their completely reducible representations $V_{G}$ and $V_{H}$. 
If $V_{H}$ is irreducible, then a description of the multiplicity of $V_{H}$ in $V_{G}$ 
regarded as a representation of $H$ by restriction is called a \textit{branching rule for} $(G,H)$. 
By Schur's lemma, the branching multiplicity is equal to the dimension of the space 
$\mathrm{Hom}_{H}(V_{H},V_{G})$, which we will call the \textit{multiplicity space}.

\subsection{}\label{def_stable}
In this paper, we shall consider branching rules of the polynomial
representations of the following pairs $(G,H)$ of complex classical
groups: $({GL}_{m}, {GL}_{n})$, $({Sp}_{2m}, {Sp}_{2n})$, $({SO}_{p}, {SO}_{q})$.
Our goal is to study branching rules for $(G,H)$ collectively in the
context of toric degenerations of spherical varieties and to obtain an
explicit description of the multiplicity space $\mathrm{Hom}_{H}(V_{H}^{\mu}, 
V_{G}^{\lambda })$ when the length $\ell (\lambda )$ of highest weight 
$\lambda $ for $G$ satisfies the following \textit{stable range condition}:

\begin{enumerate}
\item $\ell (\lambda )\leq m$ for $({GL}_{m},{GL}_{n});$

\item $\ell (\lambda )\leq n$ for $({Sp}_{2m},{Sp}_{2n})$, $({SO}_{2m},
{SO}_{2n+1})$, $({SO}_{2m+1},{SO}_{2n+1});$

\item $\ell (\lambda )<n$ for $({SO}_{2m},{SO}_{2n})$, $({SO}_{2m+1},
{SO}_{2n}).$
\end{enumerate}

We shall construct an algebra whose graded components are spanned by the
highest weight vectors of irreducible representations of $H$ appearing in
each irreducible representation of $G$.

\subsection{}
To give a slightly more detailed overview, let us consider the ring 
$\mathcal{F}_{G}$ of regular functions over $G/U_{G}$ where $U_{G}$ is a
maximal unipotent subgroup of $G$. 
This ring is called the \textit{flag algebra for} $G$, because it can be 
realized as the multi-homogeneous coordinate ring of the flag variety. 
As a $G$-module, the flag algebra $\mathcal{F}_{G}$ contains exactly one 
copy of every irreducible representation of $G$ \cite{LT79,LT85}, and in 
this context the author  studied polynomial models for $\mathcal{F}_{G}$ 
and their flat degenerations \cite{Ki08,Ki09}. 

\medskip

By highest weight theory, the $U_{H}$-invariant subspace of $V_{G}^{\lambda }$ 
consists of the highest weight vectors of irreducible representations of $H$ 
appearing in $V_{G}^{\lambda }$. 
Therefore, the $U_{H} $-invariant subalgebra of $\mathcal{F}_{G}$ leads us to 
study the branching rules for $(G,H)$ collectively:
\begin{eqnarray}
\mathcal{F}_{G}^{U_{H}} &=&\sum_{\lambda \in \widehat{G}}\left(
V_{G}^{\lambda }\right) ^{U_{H}}  \label{general branching algebra} \\
&=&\sum_{\lambda \in \widehat{G}}\sum_{\mu \in \widehat{H}}m(V_{H}^{\mu},
V_{G}^{\lambda })\left( V_{H}^{\mu }\right) ^{U_{H}}  \notag
\end{eqnarray}%
where $m(V_{H}^{\mu },V_{G}^{\lambda })$ is the multiplicity of $V_{H}^{\mu}$ 
in $V_{G}^{\lambda }$. 

Moreover, we can impose a graded structure on 
$\mathcal{F}_{G}^{U_{H}}$ so that its graded components correspond to the
multiplicity spaces:
\begin{equation*}
m(V_{H}^{\mu },V_{G}^{\lambda })\left( V_{H}^{\mu }\right) ^{U_{H}}
\cong \mathrm{Hom}_{H}(V_{H}^{\mu },V_{G}^{\lambda })
\end{equation*}
for $\left( \lambda ,\mu \right) \in \widehat{G}\times \widehat{H}$. In this
sense, we may call $\mathcal{F}_{G}^{U_{H}}$ the \textit{branching algebra for} 
$(G,H)$. This algebra was introduced by Zelobenko. 
See \cite{Ze62} and \cite{ Ze73}.

\subsection{}

Recently, Howe and his collaborators studied branching algebras for classical 
symmetric pairs, especially their toric degenerations and expressions of branching 
multiplicities in terms of Littlewood-Richardson coefficients \cite{HJLTW, HTW}.
In the cases this paper concerns, using known combinatorics of branching rules, 
we can explicitly describe the multiplicity spaces and their degenerations.
More specifically, we show that the stable range branching algebras are deformations 
of semigroup algebras of generalized semistandard tableaux or equivalently Gelfand-Tsetlin
patterns, and therefore provide a precise connection between the multiplicity 
space and the combinatorial objects which count its dimension.

Starting from combinatorial data of stable range branching
multiplicities, we shall construct an affine semigroup and its semigroup algebra
graded by the pairs of highest weights for the classical groups $G$ and $H$ listed in 
\S \ref{def_stable}. This algebra can be realized as a Hibi algebra over a distributive lattice. 
Then, by using toric deformation techniques, we lift the Hibi algebra to construct 
a polynomial model of the branching algebra for $(G,H)$. We study its finite 
presentation and standard monomial type basis. 
It turns out that there is a particular type of distributive lattices 
whose Hibi algebras can uniformly describe stable range branching algebras 
for all the pairs $(G,H)$ we consider. 

\smallskip

We remark that this Hibi algebra structure in branching problems has 
interesting counterparts in tensor product decomposition problems, which 
can be explained by reciprocity properties between 
branchings and tensor products in representation theory. 
For this direction, we refer readers to \cite{HL07, HKL, KL}.

\subsection{}

This paper is arranged as follows: In Section 2, we develop the
combinatorial tools we will use. In Section 3, we study the branching
algebra for $({GL}_{m},{GL}_{n})$ and its toric degeneration. In Section
4 and Section 5, we study the distributive lattices and affine semigroups
associated with the branching rules for $({Sp}_{2m},{Sp}_{2n})$ and $({SO}_{p},{SO}_{q})$, 
and construct the corresponding stable range branching algebras.

\medskip

\section{Combinatorics of Branchings}

This section is to prepare us the combinatorial ingredients we will use to
construct stable range branching algebras.

\subsection{}

The \textit{Gelfand-Tsetlin}(GT) \textit{poset} for ${GL}_{m}$ is the poset
\begin{equation*}
\Gamma _{m}=\left\{ x_{j}^{(i)}:1\leq i\leq m, 1\leq j\leq i\right\}
\end{equation*}
satisfying $x_{j}^{(i+1)}\geq x_{j}^{(i)}\geq x_{j+1}^{(i+1)}$ for all $i$
and $j$. The elements of $\Gamma _{m}$ can be listed in a reversed
triangular array so that $x_{j}^{(i)}$ are weakly decreasing from left to right
along diagonals as GT patterns are originally drawn \cite{GT50}. Counting
from bottom to top, we will call $x^{(r)}=(x_{1}^{(r)},x_{2}^{(r)},
\cdots,x_{r}^{(r)})$ the \textit{$r$-th row} of $\Gamma _{m}$.

\begin{definition}
\begin{enumerate}
\item For $m>n$, the GT poset for $({GL}_{m},{GL}_{n})$ is
the following subposet of $\Gamma _{m}$:
\begin{equation*}
\Gamma _{m}^{n}=\left\{ x_{j}^{(i)}\in \Gamma _{m}:n\leq i\leq m\right\}.
\end{equation*}
\item In $\Gamma _{m}^{n}$, for $m\geq k$ we define the GT\ poset of length $k$ as
\begin{equation*}
\Gamma _{m,k}^{n}=\left\{ x_{j}^{(i)} \in \Gamma _{m}^{n}:j\leq k \right\}.
\end{equation*}
\end{enumerate}
\end{definition}

For example, $\Gamma _{6,4}^{3}$ can be drawn as
\begin{equation} \label{Gamm643}
\begin{array}{ccccccccc}
x_{1}^{(6)} &  & x_{2}^{(6)} &  & x_{3}^{(6)} &  & x_{4}^{(6)} &  &  \\ 
& x_{1}^{(5)} &  & x_{2}^{(5)} &  & x_{3}^{(5)} &  & x_{4}^{(5)} &  \\ 
&  & x_{1}^{(4)} &  & x_{2}^{(4)} &  & x_{3}^{(4)} &  & x_{4}^{(4)} \\ 
&  &  & x_{1}^{(3)} &  & x_{2}^{(3)} &  & x_{3}^{(3)} & 
\end{array}%
\end{equation}

\medskip

\subsection{}\label{Sec_taborder}
Next, let us consider the set $\mathcal{L}_{m}$ of all nonempty 
subsets of $\{1,2,\cdots ,m\}$. We shall write
\begin{equation*}
I=[i_{1},\cdots ,i_{a}]
\end{equation*}
for the subset consisting of $i_{c}$ with $1\leq i_{1}<\cdots <i_{a}\leq m$.
The \textit{length} $|I|=a$ of $I$ is the number of elements in $I$.

\medskip

The following partial order $\preceq$, called the \textit{tableau order}, can be
imposed on $\mathcal{L}_{m}$: for two elements $I$ and $J$ of $\mathcal{L}_m$, we say $I\preceq
J$, if $|I| \geq |J|$ and the $c$-th smallest
element in $I$ is less than or equal to the $c$-th smallest element in 
$J$ for $1\leq c\leq |J|$. 
Then, $\mathcal{L}_{m}$ with $\preceq $ forms a distributive lattice 
whose meet $\wedge $ and join $\vee $ are, 
for $I=[i_{1},\cdots ,i_{a}]$ and $J=[j_{1},\cdots ,j_{b}]$ with $\ a\leq b$, 
\begin{eqnarray*}
I\wedge J &=&[\min (i_{1},j_{1}),\cdots ,\min (i_{a},j_{a}),i_{a+1},
\cdots,i_{b}] \\
I\vee J &=&[\max (i_{1},j_{1}),\cdots ,\max (i_{a},j_{a})].
\end{eqnarray*}
It is straightforward to check the following subposets are also distributive lattices.

\begin{definition}\label{columnDL}
\begin{enumerate}
 \item  For $m>n$, the distributive lattice $\mathcal{L}_{m}^{n}$
for $({GL}_{m},{GL}_{n})$ is the subposet of $\mathcal{L}_{m}$ consisting of
the following elements:
\begin{eqnarray*}
&&[1,2,\cdots ,r-1, r,a_{1},a_{2},\cdots ,a_{s}], \\
&&[1,2,\cdots ,r-1, r], \\
&&[a_{1},a_{2},\cdots ,a_{s}]
\end{eqnarray*}
where $r\leq n$ and $n+1\leq a_1 < \cdots < a_s \leq m$.

 \item For $k\leq m$, we let $\mathcal{L}_{m,k}^{n}$ denote 
the subposet of $\mathcal{L}_{m}^{n}$ consisting of elements of length not greater than $k$:
\begin{equation*}
\mathcal{L}_{m,k}^{n}=\left\{ I\in \mathcal{L}_{m}^{n}: 
|I|\leq k \right\}.
\end{equation*}
\end{enumerate}
\end{definition}

\subsection{}
The poset structure of $\mathcal{L}_{m,k}^{n}$ can be read from the GT poset 
$\Gamma _{m,k}^{n}$ of length $k$. For this, let us impose a partial order
on the set of order increasing subsets of $\Gamma _{m,k}^{n}$ as follows.
For two order increasing subsets $A$ and $B$ of $\Gamma _{m,k}^{n}$, we say 
$A$ is bigger than $B$, if $A\subseteq B$ as sets. Note that here we use the
reverse inclusion order on sets, because we use order increasing sets instead of 
order decreasing sets.

\begin{proposition}\label{joinirred2}
There is an order isomorphism between $\mathcal{L}_{m,k}^{n}$ and 
the set of order increasing subsets of $\Gamma _{m,k}^{n}$.
\end{proposition}

This is an easy computation similar to \cite[Theorem 3.8]{Ki08}. For each 
$I\in \mathcal{L}_{m,k}^{n}$, we define the corresponding order increasing
subset $A_{I}$ of $\Gamma _{m,k}^{n}$ as
\begin{equation}
A_I = \bigcup_{n \leq i \leq m} \left\{ x_{1}^{(i)},x_{2}^{(i)}\cdots ,x_{s_{i}}^{(i)}\right\}
\label{order increasing}
\end{equation}
where $s_{i}$ is the number of entries in $I$ less than
or equal to $i$. For example, the subset of $\Gamma_{6,4}^3$ given in (\ref{Gamm643})
corresponding to $I=[1,4,6] \in \mathcal{L}_{6,4}^3$ is
\begin{equation*}
\begin{array}{ccccccccc}
x_{1}^{(6)} &  & x_{2}^{(6)} &  & x_{3}^{(6)} &  &  &  &  \\ 
& x_{1}^{(5)} &  & x_{2}^{(5)} &  &  &  &  &  \\ 
&  & x_{1}^{(4)} &  & x_{2}^{(4)} &  &  &  & \\ 
&  &  & x_{1}^{(3)} &  &  &  & & 
\end{array}%
\end{equation*}
Then, it is straightforward to check that this correspondence
gives an order isomorphism. In fact, this Proposition gives an example of Birkhoff's
representation theorem or the fundamental theorem for finite distributive
lattices \cite[Theorem 3.4.1]{Sta97}. See \cite[\S 3.3]{Ki08} for further
details.

\smallskip

For $k \leq n$ and $d \geq 0$, we can identify $\Gamma_{m,k}^n$ with $\Gamma_{m+d,k}^{n+d}$
by shifting the $i$-th row $x^{(i)}$ up to the $(i+d)$-th row $x^{(i+d)}$ for
$n \leq i \leq m$, and then the above Proposition gives
\begin{corollary}\label{L-conversion}
For $k \leq n$ and $d \geq 0$, there is an order isomorphism between distributive 
lattices $$ \mathcal{L}_{m, k}^n \cong  \mathcal{L}_{m+d, k}^{n+d}$$
\end{corollary}

\subsection{}

A \textit{shape} or \textit{Young diagram} is a left-justified array of
boxes with weakly decreasing row lengths. We identify a shape with its
sequence of row lengths $D=(r_{1},r_{2},\cdots )$. The following example
shows the shape $D=(4,2,1)$:
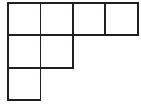
\begin{equation*}
\young(\ \ \ \ ,\ \ ,\ )
\end{equation*}
If $l$ is maximal with $r_{l}\neq 0$, then we call $l$ the \textit{length}
of $D$ and write $\ell (D)=l$. If we flip a shape $D$ over its main diagonal
that slants down from upper left to lower right, then we obtain its 
\textit{conjugate} $D^{t}$. With the previous example, we have $\ell (D)=3$ and 
$D^{t}=(4,2,1)^{t}=(3,2,1,1)$. For $F=(f_{1},f_{2},\cdots )$ and 
$D=(d_{1},d_{2},\cdots )$, if $f_{r}\geq d_{r}$ for all $r$, then we write 
$F\supseteq D$ and let $F/D$ denote the \textit{skew shape} having $F$ as its
outer shape and $D$ as its inner shape.

\subsection{}\label{tab-chain}

Consider a multiset $\{I_{1},\cdots ,I_{s}\}\subset \mathcal{L}_{m}$ with 
$|I_{c}|=l_{c}$ for each $c$. A concatenation $\mathsf{t}$ of its
elements is called a \textit{tableau}, if they are arranged so that 
$l_{c}\geq l_{c+1}$ for all $c$. The \textit{shape} $sh(\mathsf{t})$ of 
$\mathsf{t}$ is the Young diagram $(l_{1},\cdots ,l_{s})^{t}$ and the 
\textit{length} $\ell (\mathsf{t})$ of $\mathsf{t}$ is the length of its shape. 
If $\{I_{1},\cdots ,I_{s}\}$ is taken from the subposet $\mathcal{L}_{m}^{n}$,
then we shall specify the outer and inner shapes of $\mathsf{t}$.

\begin{definition}\label{def_GL-stm}
A standard tableau $\mathsf{t}$ for $({GL}_{{m}},{GL}_{{n}})$ is a multiple chain 
$$
\mathsf{t}=\left( I_{1}\preceq \cdots \preceq I_{s} \right)
$$ 
in $\mathcal{L}_{m}^{n}$. The shape $sh_{n}(\mathsf{t})$ of $\mathsf{t}$ is $F/D$ 
where%
\begin{equation*}
F =(|I_{1}|,\cdots , |I_{s}|)^{t} \hbox{ and } D =(d_{1},\cdots ,d_{n})
\end{equation*}%
and $d_{r}$ is the number of $r$'s in $\mathsf{t}$ for $1\leq r\leq n$.
\end{definition}

For example, the multiple chain $[1,2,3,6] \preceq  [1,2,5,6] \preceq  [1,2,6] \preceq  [1,4] \preceq  [5] \preceq  [5]$ in 
$\mathcal{L}_{6,4}^3$ forms a standard tableau for 
$({GL}_{{6}},{GL}_{{3}})$ of shape $(6,4,3,2)/(4,3,1)$:
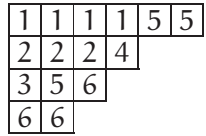
\begin{equation*}
\young(111155,2224,356,66)
\end{equation*}
which can be, after erasing $r \leq 3$, identified with the following skew semistandard tableau
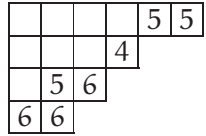
\begin{equation*}
\young(\ \ \ \ 55,\ \ \ 4,\ 56,66)
\end{equation*}

\medskip

\subsection{}\label{grading_GL}
The following set of pairs of Young diagrams will be used frequently: for 
$a\geq b$,
\begin{equation*}
\Lambda _{a,b}=\{(F,D):\ell (F)\leq a,\ell (D)\leq b,F\supseteq D\}.
\end{equation*}
We note that if $(F,D)\in \Lambda_{a,b}$, then $\ell (D)\leq \min (\ell (F),b)$.
This is because $F\supseteq D$ implies $\ell (F)\geq \ell (D)$.

\medskip

\subsection{}
Let $\mathcal{T}_{m}^{n}(F,D)$ denote the set of all standard tableaux
for $({GL}_{{m}},{GL}_{{n}})$ whose shapes are $F/D$. For each $k$ with 
$n\leq k\leq m$, we consider the following disjoint union over 
$\Lambda_{k,n}$
\begin{equation*}
\mathcal{T}_{m,k}^{n}=\bigcup_{(F,D)\in \Lambda_{k,n}}\mathcal{T}_{m}^{n}(F,D)
\end{equation*}

As illustrated by the example in \S \ref{tab-chain}, if we identify the elements 
of $\mathcal{L}_{m}^{n}$ with single-column
tableaux, then our definition of standard tableaux for $({GL}_{{m}},{GL}_{{n}})$
of shape $F/D$ agrees with the usual definition of skew semistandard Young
tableaux of shape $F/D$ with entries from $\{n+1,\cdots ,m\}$. 

By setting tableaux in the context of a finite distributive lattice (Definition \ref{def_GL-stm}), 
we can exploit an additional structure: Proposition \ref{joinirred2} leads us to study 
$\mathcal{L}_{m, k}^{n}$ in terms of the order increasing subsets of 
$\Gamma _{m, k}^{n}$, and the order increasing subsets of $\Gamma _{m, k}^{n}$ 
give rise to the order preserving maps from $\Gamma _{m, k}^{n}$ to $\{0,1\}$. 
More generally,

\begin{definition}
A GT pattern for $({GL}_{{m}},{GL}_{{n}})$ is an order preserving map from
the GT poset $\Gamma _{m}^{n}$ for $({GL}_{m},{GL}_{n})$ to the set of
non-negative integers:
\begin{equation*}
\mathsf{p}:\Gamma _{m}^{n}\rightarrow \mathbb{Z}_{\geq 0}.
\end{equation*}
The $r$-th row of $\mathsf{p}$ is 
$(\mathsf{p}(x_{1}^{(r)}),\cdots ,\mathsf{p}(x_{r}^{(r)}))$ for $n\leq r\leq m$. 
The \textit{type} of $\mathsf{p}$ is $F/D$ where $F$ and $D$ are its $m$-th 
row and the $n$-th row respectively.
\end{definition}

Note that if $\ell (F)\leq k$, then the support of every GT pattern 
$\mathsf{p}$ of type $F/D$ lies in the GT poset $\Gamma _{m,k}^{n}$ of 
length $k$. Therefore, we have GT patterns defined on $\Gamma _{m,k}^{n}$
\begin{equation*}
\mathsf{p}:\Gamma _{m,k}^{n}\rightarrow \mathbb{Z}_{\geq 0}.
\end{equation*}
Let $\mathcal{P}_{m}^{n}(F,D)$ denote the set of all GT patterns for 
$({GL}_{{m}},{GL}_{{n}})$ whose type is $F/D$. Then for each $k$ with 
$n\leq k\leq m$, we consider the following disjoint union over $\Lambda _{k,n}$:
\begin{equation}\label{GT-Patterns}
\mathcal{P}_{m,k}^{n}=\bigcup_{(F,D)\in \Lambda _{k,n}}\mathcal{P}_{m}^{n}(F,D)
\end{equation}

\subsection{}\label{polycone}

Since the sum of two order preserving maps is an order preserving map, 
$\mathcal{P}_{m,k}^{n}$ is a semigroup with function addition as its
multiplication, or more precisely a monoid with the zero function as its
identity. 
We further note that $\mathcal{P}_{m,k}^{n}$ is generated by the order preserving
maps from $\Gamma _{m,k}^{n}$ to $\{0,1\}$. Then, by identifying each GT pattern 
$\mathsf{p}$ with $(\mathsf{p}(x_{j}^{(i)}))\in \mathbb{Z}^{N}$ where $N$ is
the number of elements in $\Gamma _{m,k}^{n}$, we see that $\mathcal{P}_{m,k}^{n}$ 
can be understood as an \textit{affine semigroup}, i.e., a finitely generated semigroup which is
isomorphic to a subsemigroup of $\mathbb{Z}^{N}$ containing $0$ for some $N$
\cite{BH93}.

\smallskip

This semigroup structure on GT patterns provides a simple
bijection between $\mathcal{T}_{m,k}^{n}$ and $\mathcal{P}_{m,k}^{n}$.

\begin{proposition} \label{bijection YT-GT}
For each $(F,D) \in \Lambda_{m,n}$, there is a bijection
between $\mathcal{T}_{m}^{n}(F,D)$ and $\mathcal{P}_{m}^{n}(F,D)$.
\end{proposition}

\begin{proof}
The bijection in Proposition \ref{joinirred2} provides the bijection between 
$\mathcal{L}_{m}^{n}$ and the set of characteristic functions on order
increasing subsets of $\Gamma _{m}^{n}$. This bijection can be extended to 
multiple chains in $\mathcal{L}_{m}^{n}$ as follows. Let 
$\mathsf{t}=(I_{1}\preceq \cdots \preceq I_{c})$ be a multiple chain in 
$\mathcal{L}_{m}^{n}$ and $\mathsf{p}_{I_{r}}$ be the characteristic 
function on the order increasing set $A_{I_{r}}$ corresponding to $I_{r}$ 
given in (\ref{order increasing}) for each $r$. Then we can consider the 
following correspondence:
\begin{equation}
\mathsf{t}=(I_{1}\preceq \cdots \preceq I_{c})\mapsto \mathsf{p}_{\mathsf{t}}
= \sum\limits_{r=1}^{c}\mathsf{p}_{I_{r}}.  \label{procedure2}
\end{equation}
Since the order preserving characteristic functions over $\Gamma _{m}^{n}$
generate $\mathcal{P}_{m}^{n}$, this correspondence gives a bijection
between $\mathcal{T}_{m}^{n}(F,D)$ and $\mathcal{P}_{m}^{n}(F,D)$. For
further details, see \cite{How05, Ki08}.
\end{proof}

\subsection{}

We remark that by identifying GT patterns $\mathsf{p}$ with their images 
$(\mathsf{p}(x_{j}^{(i)}))$, our definition is equivalent to the usual definition of 
GT patterns. The correspondence between the set of semistandard tableaux 
and the set of GT patterns given in the above proposition is the same as 
the known bijection or conversion procedure (e.g., \cite[\S 8.1.2]{GW09}), which 
is usually explained by successive applications of the Pieri's rules.

\medskip

For example, a pattern $\mathsf{p} \in \mathcal{P}_{6,4}^{3}$ can be
visualized by listing its value at $x_j^{(i)} \in \Gamma_{6,4}^3$

\begin{equation}\label{pattern-exp}
\begin{array}{ccccccccc}
3 &  & 3 &  & 3 &  & 1 &  &  \\ 
& 3 &  & 3 &  & 2 &  & 0 &  \\ 
&  & 3 &  & 2 &  & 1 &  & 0\\ 
&  &  & 2 &  & 2 &  & 1& 
\end{array}%
\end{equation}
Then it is the sum of the GT patterns 
\begin{equation*}
\begin{array}{ccccccccc}
1 &  & 1 &  & 1 &  & 1 &  &  \\ 
& 1 &  & 1 &  & 1 &  & 0 &  \\ 
&  & 1 &  & 1 &  & 1 &  & 0\\ 
&  &  & 1 &  & 1 &  & 1& 
\end{array}
\ +
\begin{array}{ccccccccc}
1 &  & 1 &  & 1 &  & 0 &  &  \\ 
& 1 &  & 1 &  & 1 &  & 0 &  \\ 
&  & 1 &  & 1 &  & 0 &  & 0\\ 
&  &  & 1 &  & 1 &  & 0& 
\end{array}
\ +\begin{array}{ccccccccc}
1 &  & 1 &  & 1 &  & 0 &  &  \\ 
& 1 &  & 1 &  & 0 &  & 0 &  \\ 
&  & 1 &  & 0 &  & 0 &  & 0\\ 
&  &  & 0 &  & 0 &  & 0& 
\end{array}
\end{equation*}
corresponding to the elements $[1,2,3,6] \preceq [1,2,5] \preceq [4,5,6]$ of 
$\mathcal{L}_{6,4}^3$. This multiple chain can be identified with the following
standard tableau in $\mathcal{T}_{6,4}^3$
\begin{equation}\label{ex-ss}
\young(114,225,356,6)
\end{equation}
of shape $(3,3,3,1)/(2,2,1)$. Note that to (\ref{ex-ss}), we can apply the 
usual conversion procedure (e.g., \cite[\S 8.1.2]{GW09}) to obtain 
its corresponding pattern---by successively striking out the boxes with 
$6$, $5$, and $4$ in the tableau (\ref{ex-ss}), we obtain each row of the 
pattern (\ref{pattern-exp}).

\subsection{}\label{Hibi_mkn}
Now we study an algebra attached to $\mathcal{L}_{m,k}^{n}$ in terms of the set 
$\mathcal{T}_{m,k}^n$ of standard tableaux and the set $\mathcal{P}_{m,k}^n$ of 
GT patterns.

\begin{definition} [{\protect\cite{Hi87}}]
The \textit{Hibi algebra} $\mathcal{H}(L)$ over a finite distributive lattice $L$ 
is the quotient ring of the polynomial ring $\mathbb{C}[z_{\gamma }:\gamma \in L]$ 
by the ideal generated by 
$z_{\alpha }z_{\beta } - z_{\alpha \wedge \beta }z_{\alpha \vee \beta }$ for all 
incomparable pairs $(\alpha ,\beta )$ of $L$:
\begin{equation*}
\mathcal{H}(L) = \mathbb{C}[z_{\gamma }:\gamma \in L] / 
\langle z_{\alpha }z_{\beta } - z_{\alpha \wedge \beta }z_{\alpha \vee \beta } \rangle
\end{equation*}
\end{definition}

Let us consider the Hibi algebra over $\mathcal{L}_{m,k}^{n}$
\begin{equation*}
\mathcal{H}_{m,k}^{n}=\mathcal{H}(\mathcal{L}_{m,k}^{n}).
\end{equation*}
We shall identify the monomials $\prod_r {z_{I_r}}$ in $\mathcal{H}_{m,k}^{n}$
with the tableaux consisting of elements $I_r \in \mathcal{L}_{m,k}^{n}$. For 
example, the above tableau (\ref{ex-ss}) will be used to denote the monomial 
$$
z_{[1236]} z_{[125]} z_{[456]} \in \mathcal{H}_{6,4}^{3}
$$
Recall that standard tableaux are multiple chains in $\mathcal{L}_{m,k}^{n}$ 
(Definition \ref{def_GL-stm}). Then the following is from a general property 
of Hibi algebras \cite{Hi87, How05}.

\begin{lemma} \label{Hibi-basis}
\begin{enumerate}
\item The set $\mathcal{T}_{m,k}^{n}$ of all standard tableaux for 
$({GL}_{{m}},{GL}_{{n}})$ whose shapes are $F/D$ with $\ell (F)\leq k$ form 
a $\mathbb{C}$-basis for the Hibi algebra $\mathcal{H}_{m,k}^{n}$.

\item In particular, $\mathcal{H}_{m,k}^{n}$ is graded by $\Lambda_{k,n}$, and 
the set $\mathcal{T}_m ^n(F,D)$ of standard tableaux for $(GL_m, GL_n)$
of shape $F/D$ form a $\mathbb{C}$-basis for the $(F,D)$-graded component of 
$\mathcal{H}_{m,k}^n$.
\end{enumerate}
\end{lemma}

It is shown in \cite[Corollary 3.14]{Ki08} that the Hibi algebra over 
$\mathcal{L}_{m}$ is isomorphic to the semigroup algebra of GT patterns
defined on $\Gamma _{m}$. This fact combined with the above Lemma leads us
to study the Hibi algebra $\mathcal{H}_{m,k}^{n}$ over 
$\mathcal{L}_{m,k}^{n} $ in terms of the semigroup algebra 
$\mathbb{C}[\mathcal{P}_{m,k}^{n}]$ of the affine semigroup $\mathcal{P}_{m,k}^{n}$
given in (\ref{GT-Patterns}).

\medskip

Note that for $\mathsf{p}_1$ and $\mathsf{p}_2 \in \mathcal{P}_{m,k}^{n}$ of
types ${F_1}/{D_1}$ and ${F_2}/{D_2}$ respectively, the type of 
$(\mathsf{p}_1 + \mathsf{p}_2)$ is $(F_1 + F_2)/(D_1 + D_2)$, 
and therefore $\mathbb{C}[\mathcal{P}_{m,k}^{n}]$ is graded by pairs of shapes
\begin{equation*}
\mathbb{C}[\mathcal{P}_{m,k}^{n}]=\bigoplus_{(F,D)\in \Lambda _{k,n}}
\mathbb{C}[\mathcal{P}_{m}^{n}]_{(F,D)}
\end{equation*}
where $\mathbb{C}[\mathcal{P}_{m}^{n}]_{(F,D)}$ is the space spanned by 
$\mathcal{P}_{m}^{n}(F,D)$.

\begin{proposition}\label{DL-GT}
\begin{enumerate}
\item The Hibi algebra $\mathcal{H}_{m,k}^{n}$ over 
$\mathcal{L}_{m,k}^{n}$ is isomorphic to the semigroup algebra 
$\mathbb{C}[\mathcal{P}_{m,k}^{n}]$ of the GT patterns for $({GL}_{m},{GL}_{n})$.

\item The set $\mathcal{P}_{m}^{n}(F,D)$ of GT patterns for 
$({GL}_{m},{GL}_{n})$ of type $F/D$ is a $\mathbb{C}$-basis for the 
$(F,D)$-graded component $\mathbb{C}[\mathcal{P}_{m}^{n}]_{(F,D)}$.
\end{enumerate}
\end{proposition}

\begin{proof}
For the first statement, by Proposition \ref{joinirred2}, there is a bijection
between the set of elements $x_{I}\in \mathcal{H}_{m,k}^{n}$ for 
$I\in \mathcal{L}_{m,k}^{n}$ and characteristic functions over order increasing
subsets of $\Gamma _{m,k}^{n}$. This gives an algebra isomorphism from 
$\mathcal{H}_{m,k}^{n}$ to $\mathbb{C}[\mathcal{P}_{m,k}^{n}]$ via 
(\ref{procedure2}). See \cite{How05, Ki08} for further details. The second
statement follows from Proposition \ref{bijection YT-GT} and the above Lemma.
\end{proof}

\medskip


\section{Branching Algebras for $({GL}_{m},{GL}_{n})$}


In this section, our goal is to construct an algebra encoding branching
rules for $({GL}_{m},{GL}_{n})$ and study its toric degeneration. For later
use, we will construct a family of algebras parametrized by the length $k$ of
highest weights for ${GL}_{m}$.

\subsection{}

Recall that the set of Young diagrams $F$ with $\ell (F)\leq m$ can be used
as a labeling system of irreducible polynomial representations of ${GL}_{m}$
by identifying dominant weights $(f_{1} \geq \cdots \geq f_{m})\in \mathbb{Z}_{\geq 0}^{m}$
of ${GL}_{m}$ with Young diagrams (cf. \cite[\S 3.1.4]{GW09}). 
We let $\rho _{m}^{F}$ denote the irreducible 
polynomial representation of ${GL}_{m}$ labeled by Young diagram $F$. 

Then the branching algebra for $({GL}_{m},{GL}_{n})$ will be
graded by the set $\Lambda _{m,n}$ defined in \S \ref{grading_GL} and its
graded components will correspond to the multiplicity spaces 
$\mathrm{Hom}_{{GL}_{n}}(\rho _{n}^{D},\rho _{m}^{F})$ for 
$(F,D)\in \Lambda _{m,n}$.

\subsection{}

For Young diagrams $F=(f_{1},f_{2},\cdots )$ and $D=(d_{1},d_{2},\cdots )$,
we write $$F\sqsupseteq D$$ 
if $f_{r}\geq d_{r}\geq f_{r+1}$ for all $r$, and say  $F$ \textit{interlaces} $D$.

\begin{proposition}\label{number of Patterns}
\begin{enumerate}
\item For Young diagrams $F$ and $D$ with 
$\ell(F)\leq m$ and $\ell (D)\leq m-1$, the multiplicity of $\rho _{m-1}^{D}$ in 
$\rho _{m}^{F}$ is $1$ if $F\sqsupseteq D$, and $0$ otherwise.

\item The number of GT patterns in $\mathcal{P}_{m}^{n}(F,D)$ is equal to the
multiplicity $m(\rho _{n}^{D},\rho _{m}^{F})$ of $\rho _{n}^{D}$ in 
$\rho_{m}^{F}$.
\end{enumerate}
\end{proposition}

\begin{proof}
The first statement is known as the Pieri's rule (see, e.g., \cite[\S 8.1.1]{GW09}). 
The second statement can be obtained by applying the Pieri's
rule for $({GL}_{k+1},{GL}_{k})$ successively for $n\leq k\leq m-1$. The
multiplicity of $\rho _{n}^{D}$ in $\rho _{m}^{F}$ is equal to the number of
the set $\left\{ (E_{m-1},E_{m-2},\cdots ,E_{n+1})\right\} $ such that
\begin{equation*}
F\sqsupseteq E_{m-1}\sqsupseteq E_{m-2}\sqsupseteq \cdots \sqsupseteq
E_{n+1}\sqsupseteq D.
\end{equation*}
Note that by setting $F=E_{m}$ and $D=E_{n}$, $E_{i}$ represents the $i$-th
row of a GT pattern in $\mathcal{P}_{m}^{n}(F,D)$ for $n\leq i\leq m$.
\end{proof}

From Proposition \ref{number of Patterns} and 
Proposition \ref{bijection YT-GT}, we have

\begin{corollary}
\label{Tabnumber}For $(F,D)\in \Lambda _{m,n}$, the branching multiplicity 
$m(\rho _{n}^{D},\rho _{m}^{F})$ is equal to the number of standard tableaux
for $({GL}_{{m}},{GL}_{{n}})$ whose shapes are $F/D$.
\end{corollary}

\subsection{}

To construct a family of branching algebras for $({GL}_{m},{GL}_{n})$
parameterized by the length $k$, let us review a polynomial model for the
flag algebra. We assume $m\geq k$ and let ${GL}_{m}\times {GL}_{k}$ be
acting on the space $\mathsf{M}_{m,k} \cong \mathbb{C}^m \otimes \mathbb{C}^k$ 
of $m\times k$ complex matrices by 
\begin{equation}
(g_{1},g_{2})\cdot Q=(g_{1}^{t})^{-1}Q g_{2}^{-1}  \label{doubleaction}
\end{equation}
for $g_{1}\in {GL}_{m}$, $g_{2}\in {GL}_{k}$, and $Q\in \mathsf{M}_{m,k}$.
Then under the ${GL}_{m}\times {GL}_{k}$ action, the coordinate ring 
$\mathbb{C}[\mathsf{M}_{m,k}]$ of $\mathsf{M}_{m,k}$ has the following
decomposition:
\begin{equation*}
\mathbb{C}[\mathsf{M}_{m,k}]=\sum\limits_{\ell (F)\leq k}
\rho_{m}^{F}\otimes \rho _{k}^{F}
\end{equation*}
where the summation is over $F$ with length not more than $k$.
This result is known as ${GL}_{m}$-${GL}_{k}$ \textit{duality} 
(e.g., \cite{GW09, How95}). If $U_{k}$ is the subgroup of ${GL}_{k}$ 
consisting of upper triangular matrices with $1$'s on the diagonal, then by 
taking $U_{k}\cong 1\times U_{k}$ invariants, we have
\begin{equation}
\mathbb{C}[\mathsf{M}_{m,k}]^{U_{k}}=\sum\limits_{\ell (F)\leq k}\rho
_{m}^{F}\otimes (\rho _{k}^{F})^{U_{k}}  \notag
\end{equation}

\subsection{}

This representation decomposition turns out to be compatible with the
multiplicative structure of the algebra. Since the diagonal subgroup $A_{k}$
of ${GL}_{k}$ normalizes $U_{k}$, $\mathbb{C}[\mathsf{M}_{m,k}]^{U_{k}}$ is
stable under the action of $A_{k}$. Note that by highest weight theory
(e.g., \cite[\S 3.2.1 and \S 12.1.3]{GW09}), $(\rho _{k}^{F})^{U_{k}}$ is the one
dimensional space spanned by a highest weight vector of $\rho _{k}^{F}$, and 
$A_{k}$ acts on $(\rho _{k}^{F})^{U_{k}}$ by the character
\begin{equation*}
\phi_F(diag(a_{1},\cdots ,a_{k})) = a_{1}^{f_{1}}\cdots a_{k}^{f_{k}}
\end{equation*}
given by Young diagram $F=(f_{1},f_{2},\cdots ,f_{k})$. 
Thus, $\rho _{m}^{F}\simeq \rho _{m}^{F}\otimes (\rho_{k}^{F})^{U_{k}}$ is the space 
of $A_{k}$-eigenvectors of weight $\phi_F$ in $\mathbb{C}[\mathsf{M}_{m,k}]^{U_{k}}$ and 
the $\mathbb{C}$-algebra $\mathbb{C}[\mathsf{M}_{m,k}]^{U_{k}}$ is graded by the 
semigroup $\hat{A}_{k}^{+}$ of dominant polynomial weights for ${GL}_{k}$, or equivalently 
the subsemigroup $\hat{A}_{k}^{+}\subset \hat{A}_{m}^{+}$ of dominant weights for ${GL}_{m}$: 
\begin{eqnarray}
\mathbb{C}[\mathsf{M}_{m,k}]^{U_{k}} &=&\sum\limits_{\ell (F)\leq k}
\rho_{k}^{F}  \label{grading1} \\
\rho _{m}^{F_1}\cdot \rho _{m}^{F_2} &\subseteq &\rho_{m}^{F_1 + F_2}  \notag
\end{eqnarray}
where we identify $(r_1, \cdots, r_k) \in \mathbb{Z}_{\geq 0}^k$ with 
$(r_1, \cdots, r_k, 0, \cdots, 0) \in \mathbb{Z}_{\geq 0}^m$.
\subsection{}

A finite presentation of $\mathbb{C}[\mathsf{M}_{m,k}]^{U_{k}}$ 
in terms of generators and relations is well known---all the $U_{k}$-invariant minors on 
$\mathsf{M}_{m,k}$ form a generating set and they satisfy the Pl\"{u}cker relations. 
To explain more details, let us consider a subposet $\mathcal{L}_{m,k}=\mathcal{L}_{m,k}^1$ 
of $\mathcal{L}_{m}$ consisting of elements $I=[i_{1},i_{2},\cdots ,i_{r}]$ such that 
$|I| \leq k$ (cf. Definition \ref{columnDL}). 

For each $Q \in \mathsf{M}_{m,k}$, we let $\delta _{I}(Q)$ denote the determinant of
the submatrix of $Q=(t_{a,b})$ obtained by taking the
$i_{1},i_{2},\cdots ,i_{r}$-th rows and the $1,2,\cdots ,r$-th columns:
\begin{equation}
\delta _{I}(Q)=\det 
\begin{bmatrix}
t_{i_{1}1} & t_{i_{1}2} & \cdots  & t_{i_{1}r} \\ 
t_{i_{2}1} & t_{i_{2}2} & \cdots  & t_{i_{2}r} \\ 
\vdots  & \vdots  & \ddots  & \vdots  \\ 
t_{i_{r}1} & t_{i_{r}2} & \cdots  & t_{i_{r}r}%
\end{bmatrix}
\label{determinant}
\end{equation}
\begin{definition}
A product $\delta _{I_{1}}\delta _{I_{2}}\cdots \delta _{I_{r}}$ is
called a \textit{standard monomial} (or \textit{$GL_m$ standard monomial}), 
if its indexes form a multiple
chain $\mathsf{t}=(I_{1}\preceq I_{2}\preceq \cdots \preceq I_{r})$ in 
$\mathcal{L}_{m,k}$ and write
\begin{equation*}
\Delta _{\mathsf{t}}=\delta _{I_{1}}\delta _{I_{2}}\cdots \delta _{I_{r}}.
\end{equation*}
\end{definition}

Then we define the \textit{shape} of a standard monomial $\Delta _{\mathsf{t}}$
to be the shape of $\mathsf{t}$, i.e., $(|I_1|, |I_2|, \cdots, |I_r|)^t$.

\begin{proposition}[{\protect\cite[pp.233,236]{GL01}}]\label{straightening1}
\begin{enumerate}
\item For $I,J\in \mathcal{L}_{m,k}$, the product $\delta _{I}\delta _{J}
\in \mathbb{C}[\mathsf{M}_{m,k}]^{U_{k}}$ can
be uniquely expressed as a linear combination of standard monomials
\begin{equation}
\delta _{I}\delta _{J}=\sum_{r}c_{r}\delta _{S_{r}}\delta _{T_{r}}
\label{straightening}
\end{equation}
where, for each $r$ with $c_r \ne 0$, $S_{r}\preceq T_{r}$ in $\mathcal{L}_{m,k}$ 
and  $S_{r}\dot{\cup} T_{r} = I \dot{\cup} J$ as sets.

\item On the right hand side, $\delta _{I\wedge J}\delta _{I\vee J}$ appears
with coefficient $1$, and $S_{r}\preceq I\wedge J$ and $I\vee J\preceq
T_{r} $ for all $r$ with $c_r \ne 0$. Moreover, for each $(S_{r},T_{r})\neq (I\wedge J,
I \vee J)$, let $h$ be the smallest integer such that the sum $s$ of the $h$-th
entries of $S_{r}$ and $T_{r}$ is different from the sum $s_{0}$ of the 
$h$-th entries of $I$ and $J$. Then $s>s_{0}$.
\end{enumerate}
\end{proposition}

By applying the straightening relations (\ref{straightening}), we can find a 
$\mathbb{C}$-basis for 
$\mathbb{C}[\mathsf{M}_{m,k}]^{U_{k}}$. The following is well known. See, for
example, \cite{BH93,DEP82,GL01,Hod43}. For this particular form, see 
\cite[Theorem 4.5, Remark 4.6]{Ki08}.

\begin{proposition}
Standard monomials $\Delta _{\mathsf{t}}$ associated with multiple chains 
$\mathsf{t}$ in $\mathcal{L}_{m,k}$ form a $\mathbb{C}$-basis for 
$\mathbb{C}[\mathsf{M}_{m,k}]^{U_{k}}$. More precisely, standard monomials 
$\Delta _{\mathsf{t}}$ with $sh(\mathsf{t})=F$ form a weight basis for the 
${GL}_{m}$ irreducible representation 
$\rho _{m}^{F}\subset \mathbb{C}[\mathsf{M}_{m,k}]^{U_{k}}$ with highest weight $F$.
\end{proposition}

\smallskip

We specify the following properties of the standard monomial expression of 
$\delta_I \delta_J$ for $I,J \in \mathcal{L}_{m,k}^n$ of length not more than $k$, 
which can be easily derived from the above Proposition.

\begin{corollary}\label{straightening123} 
Let $I$ and $J$ be incomparable elements in $\mathcal{L}_{m,k}^n$ with $|I| \geq |J|$. Consider the 
standard monomial expression of the product $\delta_I \delta_J$ given in (\ref{straightening}). Let 
us denote the standard tableau $S_r \preceq T_r$ by $\mathsf{t}_r$. Then, for each $r$ with nonzero
$c_r$,
\begin{enumerate}
\item the shape $sh_n(\mathsf{t}_r)$ is $F/D$ where $F=(|I|, |J|)^t$ and 
$D=(d_1, d_2, \cdots)$ where $d_h$ is the number of $h$'s in 
the disjoint union $I \dot{\cup} J$ for $1 \leq h \leq n$;

\item all the entries in the $h$-th row of ${\mathsf{t}_r}$ are bigger than or equal to 
$h$ for $1 \leq h \leq min(n,|I|)$;

\item if we denote the numbers of entries less than or equal to $h$ in $S_r$ and $T_r$ by 
$\alpha_h$ and $\beta_h$ respectively, then $\alpha_{h} + \beta_{h} \leq {2h}$ 
for $1 \leq h \leq min(n,|I|)$.
\end{enumerate}
\end{corollary}

\begin{example}
For $I=[1,2,5,6]$ and $J=[1,3,4]$ from $\mathcal{L}_{6,4}^2$, we have
\begin{eqnarray*}
\delta_{[1256]} \delta_{[134]} &=& \delta_{[1246]} \delta_{[135]}  - \delta_{[1236]} \delta_{[145]}\\
 && + \delta_{[1235]} \delta_{[146]} - \delta_{[1245]} \delta_{[136]} - \delta_{[1234]} \delta_{[156]} 
\end{eqnarray*}
Note that $sh_n(\mathsf{t}_r)=(2,2,2,1)/(2,1)$ for all the terms $\mathsf{t}_r$ 
on the right hand side.
\end{example}

\subsection{}
Let $m>n$. To consider the branching rules for $({GL}_{m},{GL}_{n})$, we use
the following embedding of ${GL}_{n}$ in ${GL}_{m}$: for $X\in {GL}_{n}$,
\begin{equation*}
\left[ 
\begin{array}{cc}
X & 0 \\ 
0 & I%
\end{array}%
\right] \in {GL}_{m}
\end{equation*}
where $I$ is the $(m-n) \times (m-n)$ identity matrix and $0$'s are the zero
matrices of proper sizes.

From (\ref{grading1}), by taking $U_{n}$-invariants, we have 
\begin{eqnarray}
\mathbb{C}[\mathsf{M}_{m,k}]^{U_{n}\times U_{k}} &=&\sum\limits_{\ell
(F)\leq k}\left( \rho _{m}^{F}\right) ^{U_{n}}  \label{UU-decomposition} \\
&=&\sum\limits_{\ell (F)\leq k}\sum\limits_{D}m(\rho _{n}^{D},\rho
_{m}^{F})\left( \rho _{n}^{D}\right) ^{U_{n}}  \notag
\end{eqnarray}
where $m(\rho _{n}^{D},\rho _{m}^{F})$ is the multiplicity of $\rho _{n}^{D}$
appearing in $\rho _{m}^{F}$, and $\left( \rho _{n}^{D}\right) ^{U_{n}}$ is
the one-dimensional space spanned by a highest weight vector of $\rho_{n}^{D}$.

\begin{definition}\label{Bmkn}
For $m\geq k$, the length $k$ branching algebra for $({GL}_{m},{GL}_{n})$ is
the $(U_{n}\times U_{k})$-invariant ring of $\mathbb{C}[\mathsf{M}_{m,k}]$
\begin{equation*}
\mathcal{B}_{m,k}^{n}=\mathbb{C}[\mathsf{M}_{m,k}]^{U_{n}\times U_{k}}
\end{equation*}
\end{definition}

\medskip

\subsection{}

Note that for $I\in \mathcal{L}_{m,k}^{n}$ all the minors $\delta _{I}$ are
invariant under the subgroup $U_{n}\times U_{k}$ of ${GL}_{n}\times {GL}_{k}$
with respect to the action (\ref{doubleaction}). In fact, the length $k$ branching algebra 
for $({GL}_{m},{GL}_{n})$ is generated by $\{\delta_{I}:I\in \mathcal{L}_{m,k}^{n}\}$.

\begin{theorem}\label{BB-standard}
For each $k$ with $m\geq k$, the branching algebra 
$\mathcal{B}_{m,k}^{n}$ for $({GL}_{m},{GL}_{n})$ is graded by 
$\Lambda_{k,n} $
\begin{equation*}
\mathcal{B}_{m,k}^{n}=\bigoplus_{(F,D)\in \Lambda _{k,n}}\mathcal{B}_{m,k}^{n}(F,D)
\end{equation*}
and the standard monomials $\Delta _{\mathsf{t}}$ for 
$\mathsf{t}\in \mathcal{T}_{m}^{n}(F,D)$ form a $\mathbb{C}$-basis of 
the $(F,D)$-graded component $\mathcal{B}_{m,k}^{n}(F,D)$.
\end{theorem}
\begin{proof}
For $I\in \mathcal{L}_{m,k}^{n}$, the determinant functions $\delta _{I}$ as
elements of $\mathbb{C}[\mathsf{M}_{m,k}]^{U_{k}}$ satisfy the relations 
(\ref{straightening}) and by keeping track of the entries of $I$ and $J$ in
this relation, we can easily see that all $S_{r}$ and $T_{r}$ appearing on
the right hand side of (\ref{straightening}) are elements of 
$\mathcal{L}_{m,k}^{n}$ and $sh_{n}(S_{r} \preceq T_{r})$'s are the same for all $r$ as 
in the first statement of Corollary \ref{straightening123}. By applying
these relations repeatedly, we can express every monomial in 
$\{\delta _{I}: I\in \mathcal{L}_{m,k}^{n}\}$ as a linear combination of standard 
monomials of the same shape. In particular, the algebra $\mathcal{B}_{m,k}^n$ is graded by 
the shapes $sh_n(\mathsf{t}) \in \Lambda_{k,n}$ of standard monomials for $(GL_m,GL_n)$.
Now, it is enough to show that for each shape $F/D$ with 
$(F,D)\in \Lambda _{k,n}$, the number of standard monomials 
$\Delta _{\mathsf{t}}$ for $\mathsf{t}\in \mathcal{T}_{m}^{n}(F,D)$ is equal to the
multiplicity of $\rho _{n}^{D}$ in $\rho _{m}^{F}$, which is Corollary 
\ref{Tabnumber}.
\end{proof}

Note that the standard monomials $\Delta _{\mathsf{t}}$ for 
$\mathsf{t}\in \mathcal{T}_{m}^{n}(F,D)$ are invariant under the action of $U_{n}$ and
scaled by the character $\phi_D$ under the action of the diagonal subgroup of $GL_{n}$:

\begin{eqnarray}\label{U-invariant-weight}
diag(a_{1},\cdots ,a_{n})\cdot \Delta _{\mathsf{t}} &=&
\phi_D \left( diag(a_{1},\cdots ,a_{n}) \right) \Delta _{\mathsf{t}} \\
&=&(a_{1}^{d_{1}}\cdots a_{n}^{d_{n}}) \Delta _{\mathsf{t}} \notag
\end{eqnarray}
for $D=(d_{1},\cdots ,d_{n})$. This shows that standard monomials 
$\Delta _{\mathsf{t}}$ for $\mathsf{t}\in \mathcal{T}_{m}^{n}(F,D)$ are the 
highest weight vectors of the copies of $\rho _{n}^{D}$ in $\rho _{m}^{F}$.
Accordingly, we have

\begin{proposition}
The standard monomials $\Delta_{\mathsf{t}}$ with 
$\mathsf{t} \in \mathcal{T}_m^n (F,D)$, as
$\mathbb{C}$-basis elements of $\mathcal{B}_{m,k}^{n}(F,D)$, are the highest weight 
vectors of the copies of $\rho _{n}^{D}$ in $\rho _{m}^{F}$. Therefore, we have
\begin{equation*}
\mathcal{B}_{m,k}^{n}(F,D)\cong \mathrm{Hom}_{{GL}_{n}}
(\rho _{n}^{D}, \rho_{m}^{F}).
\end{equation*}
\end{proposition}

\smallskip

\subsection{}

Toric degenerations of the branching algebras $\mathcal{B}_{m,k}^{n}$ can be
induced by the same methods used for the case of the flag algebra 
$\mathbb{C}[\mathsf{M}_{m,k}]^{U_{k}}$ in the literature, for example, 
\cite{GL96,Ki08,KM05, MS05, St95}. See also \cite[Theorem 1]{Vin95}, for the properties 
of the algebra of polynomials on a semisimple algebraic group and its associated 
graded algebra. 

\begin{theorem}
\label{Deformation}The length $k$ branching algebra $\mathcal{B}_{m,k}^{n}$
for $({GL}_{m},{GL}_{n})$ is a flat deformation of the Hibi algebra 
$\mathcal{H}_{m,k}^{n}$ over $\mathcal{L}_{m,k}^{n}$.
\end{theorem}

\begin{proof}
Let us impose a filtration on $\mathcal{B}_{m,k}^{n}$ by giving the
following weight on each monomials. Fix an integer $N$ greater than $2m$,
and then define the weight of $I=[i_{1},\cdots ,i_{a}]\in \mathcal{L}_{m,k}^{n}$ as
\begin{equation}
wt(I)=\sum_{r\geq 1}i_{r}N^{m-r}.  \label{def weight}
\end{equation}
The weight of a standard tableau $\mathsf{t}$ consisting of $I_{c}$ is
defined to be the sum of individual weights, i.e., 
$wt(\mathsf{t})=\sum_{c}wt(I_{c})$. Then we can define a $\mathbb{Z}$-filtration 
$\mathsf{F}^{wt}=\{\mathsf{F}_{d}^{wt}\}$ on 
$\mathcal{B}_{m,k}^{n}=\mathbb{C}[\mathsf{M}_{m,k}]^{U_{n}\times U_{k}}$ with respect to 
the weight $wt$. Set $\mathsf{F}_{d}^{wt}(\mathcal{B}_{m,k}^{n})$ to be the space spanned by
\begin{equation*}
\left\{ \Delta _{\mathsf{t}}:wt(\mathsf{t})\geq d\right\} .
\end{equation*}
The filtration $\mathsf{F}^{wt}$ is well defined, since every product 
$\prod \delta _{I_{c}}$ can be expressed as a linear combination of standard
monomials with bigger weights by Proposition \ref{straightening1}. For all pairs 
$A,B\in \mathcal{L}_{m,k}^{n}$, since $wt(A)+wt(B)=wt(A\wedge B)+wt(A\vee B)$,
 $\delta _{A}\delta _{B}$ and $\delta _{A\wedge B}\delta _{A\vee B}$ belong
to the same associated graded component. Therefore, we have 
$s_{A}\cdot_{gr}s_{B}=s_{A\wedge B}\cdot _{gr}s_{A\vee B}$ where $s_{C}$ are elements
corresponding to $\delta _{C}$ in the associated graded ring 
$\mathsf{gr}^{wt}(\mathcal{B}_{m,k}^{n})$ of $\mathcal{B}_{m,k}^{n}$ with respect to 
the filtration $\mathsf{F}^{wt}$. Then it is straightforward to show that the
associated graded ring $\mathsf{gr}^{wt}(\mathcal{B}_{m,k}^{n})$ forms the
Hibi algebra over $\mathcal{L}_{m,k}^{n}$. From a general property of the
Rees algebras (e.g., \cite{AB04}), the Rees algebra $\mathcal{R}^{t}$ of 
$\mathcal{B}_{m,k}^{n}$ with respect to $\mathsf{F}^{wt}$:
\begin{equation*}
\mathcal{R}^{t}=\bigoplus_{d\geq 0}\mathsf{F}_{d}^{wt}(\mathcal{B}_{m,k}^{n})t^{d}
\end{equation*}
is flat over $\mathbb{C}[t]$ with its general fiber isomorphic to 
$\mathcal{B}_{m,k}^{n}$ and special fiber isomorphic to the associated graded ring
which is $\mathcal{H}_{m,k}^{n}$.
\end{proof}

We remark that $Spec(\mathcal{H}_{m,k}^{n})$ is an affine toric variety in the
sense of \cite{St95}. Then, the rational polyhedral cone corresponding to the affine 
toric variety and the integral points therein can be realized from 
our description of the affine semigroup $\mathcal{P}_{m,k}^n$ given at 
the beginning of \S \ref{polycone}.

\medskip


\section{Stable Range Branching Algebra for $({Sp}_{2m},{Sp}_{2n})$}


In this section, starting from combinatorial descriptions of stable range branching rules, 
we study the affine semigroup algebra and its associated Hibi algebra for $({Sp}_{2m},{Sp}_{2n})$.
Then we construct an explicit model for the stable range branching algebra. Along with these,
we also show that these algebraic objects are isomorphic to their $(GL_{2m}, GL_{2n})$ counterparts 
with a proper length condition.

\smallskip

Recall that we can label irreducible rational representation of ${Sp}_{2m}$, after 
identifying dominant weights with Young diagrams, by Young diagrams with less than or equal
to $m$ rows (cf. \cite[\S 3.1.4]{GW09}). We let $\tau_{2m}^{F}$ denote the irreducible 
representation of ${Sp}_{2m}$ labeled by Young diagram $F$.

\subsection{}\label{sp-section1}

Let $J_{m}=(j_{a,b})$ be the $m\times m$ matrix with $j_{a,m+1-a}=1$ for 
$1\leq a\leq m$ and $0$ otherwise. Then we define the symplectic group 
${Sp}_{2m}$ of rank $m$ as the subgroup of ${GL}_{2m}$ preserving the skew
symmetric bilinear form on $\mathbb{C}^{2m}$ induced by 
\begin{equation*}
\left[ 
\begin{array}{cc}
0 & J_{m} \\ 
-J_{m} & 0
\end{array}
\right] .
\end{equation*}

Note that, for the elementary basis $\{e_{i}\}$ of the space $\mathbb{C}^{2m}$, 
$e_{j}$ and $e_{2m+1-j}$ make an isotropic pair for $1\leq j\leq m$ with respect 
to this bilinear form. Also, the subgroup of upper triangular matrices with $1$'s on 
the diagonal can be taken as a maximal unipotent subgroup of ${Sp}_{2m}$. We will
denoted it by $U_{{Sp}_{2m}}$.

\medskip

For $n<m$, we identify ${Sp}_{2n}$ with the subgroup of ${Sp}_{2m}$
preserving the skew symmetric bilinear form restricted to the subspace of 
$\mathbb{C}^{2m}$ spanned by $$\left\{ e_{a},e_{2m+1-a}:1\leq a\leq n \right\}.$$ 
Then ${Sp}_{2n}$ can be embedded in ${Sp}_{2m}$ as follows.
\begin{equation} \label{Sp embedding}
\left[ 
\begin{array}{cc}
X & Y \\ 
Z & W%
\end{array}%
\right] \mapsto \left[ 
\begin{array}{ccc}
X & 0 & Y \\ 
0 & I & 0 \\ 
Z & 0 & W%
\end{array}%
\right] 
\end{equation}
where $X,Y,Z,W$ are $n \times n$ matrices, $I$ is the $2(m-n) \times 2(m-n)$
identity matrix, and $0$'s are the zero matrices of proper sizes.

\subsection{}

In order to construct an affine semigroup encoding stable range branching rules for
$({Sp}_{2m},{Sp}_{2n})$, we review the following combinatorial description of 
branching multiplicities.

\begin{lemma}[{\protect\cite[Theorem 8.1.5]{GW09}}] \label{Sp_branching}
For Young diagrams $F$ and $D$ with $\ell (F)\leq m$ and 
$\ell (D)\leq m-1$, the multiplicity of $\tau _{2(m-1)}^{D}$ in $\tau
_{2m}^{F}$ as a ${Sp}_{2(m-1)}$ representation is equal to the number of
Young diagrams $E$ satisfying the interlacing condition $F\sqsupseteq
E\sqsupseteq D$.
\end{lemma}

For example, if $F=(5,3, 3, 2, 1)$ and $D=(4,3,2,2)$, then the multiplicity of 
$\tau_{8}^{D}$ in $\tau_{10}^{F}$ is equal to the number of 
$E=(e_1, e_2, \cdots, e_5)$ in 
\begin{equation*}
\begin{array}{cccccccccc}
5 &  & 3 &  & 3 &  & 2 &  & 1 &  \\ 
   & e_{1} &  & e_{2} &  & e_{3} &  & e_{4} &  & e_{5} \\ 
&  & 4 &  & 3 &  & 2 &  & 2 &  %
\end{array}%
\end{equation*}
so that the entries are weakly decreasing from left to right along diagonals.

\smallskip

Note that this branching is not multiplicity free and rather similar to the
two-step branchings for the general linear groups. To obtain a description
of the multiplicity spaces for $({Sp}_{2m},{Sp}_{2n})$, we can simply
iterate the above lemma. Because of the length condition 
$\ell (E_{k})\leq k$ of ${Sp}_{2k}$ representations $\tau _{2k}^{E_{k}}$ 
for $n\leq k\leq m$, it will be quite different from the $(GL_{2m}, GL_{2n})$ case 
(Proposition \ref{number of Patterns}). Within the stable 
range $\ell (F)\leq n$, however, we have exactly the same description.

In the previous example, if we set $F=(5,3, 3, 2, 0)$ so that $\ell(F)=4$, then 
the multiplicity of $\tau_{8}^{D}$ in $\tau_{10}^{F}$ is equal to the number of 
$E=(e_1, e_2, \cdots, e_5)$ in 
\begin{equation*}
\begin{array}{cccccccccc}
5 &  & 3 &  & 3 &  & 2 &  & 0 &  \\ 
   & e_{1} &  & e_{2} &  & e_{3} &  & e_{4} &  & e_{5} \\ 
&  & 4 &  & 3 &  & 2 &  & 2 &  %
\end{array}%
\end{equation*}
and the interlacing condition makes $e_5 =0$. Therefore the multiplicity
of $\tau_{8}^{D}$ in $\tau_{10}^{F}$ is equal to the multiplicity of
the $GL_8$ representation $\rho_{8}^{D}$ in the $GL_{10}$ representation $\rho_{10}^F$.

\begin{remark}
\begin{enumerate}
\item For complete GT patterns for $Sp_{2m}$, we refer to
\cite{Kir88} and \cite{Pr94}. See also \cite{Ki08} for their ring theoretic interpretation. 

\item The branching algebra for $(Sp_{2m}, Sp_{2m-2})$ without any restriction on 
the length of representations has interesting algebraic and combinatorial properties 
with an extra structure from the action of $SL_2 \times \cdots \times SL_2$. 
For this, we refer to \cite{KY11}.
\end{enumerate}
\end{remark}

\subsection{}
Recall that $\mathcal{P}_{2m}^{2n}(F,D)$ is the set of all $GT$ patterns for 
$({GL}_{2m},{GL}_{2n})$ whose types are $F/D$. Within the stable range 
$\ell(F)\leq n$, $F\supseteq D$ implies $\ell (D)\leq n$, and therefore the
support of every GT pattern in $\mathcal{P}_{2m}^{2n}(F,D)$ lies in the GT
poset $\Gamma _{2m,n}^{2n}$ of length $n$:
\begin{equation*}
\begin{array}{cccccccccc}
x_{1}^{(2m)} &  & x_{2}^{(2m)} &  & \cdots &  & x_{n}^{(2m)} &  &  &  \\ 
& x_{1}^{(2m-1)} &  & x_{2}^{(2m-1)} &  & \cdots &  & x_{n}^{(2m-1)} &  & \\ 
&  & \ddots &  & \ddots &  & \cdots &  & \ddots &  \\ 
&  &  & x_{1}^{(2n)} &  & x_{2}^{(2n)} &  & \cdots &  & x_{n}^{(2n)} %
\end{array}%
\end{equation*}

\begin{proposition}\label{Sp_counting}
Let $F$ and $D$ be Young diagrams with $F\supseteq D$ and 
$\ell (F)\leq n$. Then the branching multiplicity 
$m(\tau _{2n}^{D},\tau_{2m}^{F})$ is equal to the number of elements in 
$\mathcal{P}_{2m}^{2n}(F,D) $, and therefore it is equal to the number of elements 
in $\mathcal{T}_{2m}^{2n}(F,D)$.
\begin{proof}
From Lemma \ref{Sp_branching}, by using the same argument used to prove 
(2) of Proposition \ref{number of Patterns}, the set $\mathcal{P}_{2m}^{2n}(F,D)$ of 
GT patterns of shape $F/D$ counts the multiplicity of $\tau_{2n}^{D}$ in $\tau _{2m}^{F}$. 
The last statement follows from Proposition \ref{bijection YT-GT}.
\end{proof}
\end{proposition}

We call the affine semigroup $\mathcal{P}_{2m,n}^{2n}$, defined in 
(\ref{GT-Patterns}), the \textit{semigroup for} $({Sp}_{2m},{Sp}_{2n})$ and call
its associated semigroup algebra $\mathbb{C}[\mathcal{P}_{2m,n}^{2n}]$ the 
\textit{semigroup algebra for} $({Sp}_{2m},{Sp}_{2n})$. Then it is graded by 
$\Lambda _{n,n}$ defined in \S \ref{grading_GL}.
\begin{equation*}
\mathbb{C}[\mathcal{P}_{2m,n}^{2n}]=\bigoplus_{(F,D)\in \Lambda _{n,n}}
\mathbb{C}[\mathcal{P}_{2m}^{2n}]_{(F,D)}
\end{equation*}

\subsection{}\label{Sp-distlattice}

To define tableaux and standard monomials for the symplectic groups,
we shall use the following ordered letters:
\begin{equation}\label{Sp-alphabet}
\langle 2m \rangle = \left\{ u_1 < v_1 < u_2 < v_2 < \cdots < u_m < v_m \right\}.
\end{equation}

If we let $\mathcal{L}\langle 2m \rangle$ denote the set of all non-empty subsets $J$ of 
$\langle 2m \rangle$, then on $\mathcal{L}\langle 2m \rangle$ we can impose the tableau 
order $\preceq$, as it is done in \S \ref{Sec_taborder} for $\mathcal{L}_{2m}$, through 
the bijection 
\begin{equation}\label{iota}
\iota(u_c)= 2c-1 \hbox{ and } \iota(v_c)=2c
\end{equation}
for $1 \leq c \leq m$. Then $\mathcal{L}\langle 2m \rangle$ is a distributive lattice 
isomorphic to $\mathcal{L}_{2m}$. 

\medskip

For $m > n$, we consider the subposet 
$\mathcal{L}\langle n, 2m \rangle$ of $\mathcal{L}\langle 2m \rangle$ with all 
the elements $J \subset \langle 2m \rangle$ of the forms 
\begin{eqnarray}\label{L_Sp_elements}
&&[u_1,u_2,\cdots ,u_c,y_{1},y_{2}, \cdots , y_{s}], \\
&&[u_1,u_2,\cdots ,u_c],  \notag \\
&&[y_{1},y_{2},\cdots ,y_{s}] \notag
\end{eqnarray}%
where $c \leq n$ and $u_{n+1} \leq y_1 < y_2 < \cdots < y_s \leq v_m$. In particular,
if $u_c \in J$ for $c \leq n$, then $\{ u_h: 1 \leq h \leq c \} \subset J$.

\smallskip

Now, for $k \leq n$, let $\mathcal{L}\langle n, 2m \rangle_k$ be the subposet of
$\mathcal{L}\langle n, 2m \rangle$ consisting of $J \in \mathcal{L}\langle n, 2m \rangle$
with $|J| \leq k$. Then, through the map (\ref{iota}), it is straightforward to see that 
$\mathcal{L}\langle n, 2m \rangle_k$ is isomorphic to
$\mathcal{L}_{2m-n,k}^{n}$ given in Definition \ref{columnDL}, and therefore 
isomorphic to $\mathcal{L}_{2m, k}^{2n}$ by Corollary \ref{L-conversion}.

\begin{definition}
\begin{enumerate}
\item The distributive lattice for $(Sp_{2m}, Sp_{2n})$ is 
\begin{eqnarray*}
\mathcal{L}_{Sp} &=& \mathcal{L}\langle n, 2m \rangle_n \\
&\cong& \mathcal{L}_{2m, n}^{2n}
\end{eqnarray*}
\item The Hibi algebra for $(Sp_{2m}, Sp_{2n})$, denoted by $\mathcal{H}_{Sp}$, 
is the Hibi algebra over the distributive lattice $\mathcal{L}_{Sp}$.
\end{enumerate}
\end{definition}

Note that from $\mathcal{L}_{Sp} \cong \mathcal{L}_{2m, n}^{2n}$, 
the Hibi algebra $\mathcal{H}_{Sp}$ for $(Sp_{2m}, Sp_{2n})$ is isomorphic to 
$\mathcal{H}^{2n}_{2m,n}$. Then from Proposition \ref{DL-GT} for $(GL_{2m}, GL_{2n})$
we have

\begin{corollary}\label{Sp_Hibi_iso}
The Hibi algebra for $(Sp_{2m}, Sp_{2n})$ is isomorphic to 
the semigroup algebra for $(Sp_{2m}, Sp_{2n})$:
\begin{equation*}
\mathcal{H}_{Sp} \cong \mathbb{C}[\mathcal{P}_{2m,n}^{2n}]
\end{equation*}
\end{corollary}

\subsection{}
Next, we define standard tableaux for $(Sp_{2m}, Sp_{2n})$.

\begin{definition}\label{BrSp_standardtab}
\begin{enumerate}
\item A standard tableau $\mathsf{t}$ for $({Sp}_{2m},{Sp}_{2n})$ is a multiple
chain  in $\mathcal{L}_{Sp}$: $$\mathsf{t}=\left( I_{1}\preceq \cdots \preceq I_{s} \right).$$ 
\item The shape $sh_{n}(\mathsf{t})$ of a standard tableau $\mathsf{t}$ for 
$({Sp}_{2m},{Sp}_{2n})$ is $F/D$ where 
$$F=(|I_{1}|,\cdots ,|I_{s}|)^{t} \hbox{ and } D=(d_{1},\cdots ,d_{n})$$
with $d_{r}$ being the number of $u_h$'s in $\mathsf{t}$ for $1\leq h\leq n$.
\end{enumerate}
\end{definition}

We write $\mathcal{T}_{Sp}(F,D)$ for the set of all standard
tableaux for $({Sp}_{2m}, {Sp}_{2n})$ whose shapes are $F/D$, and consider
the disjoint union
\begin{equation*}
\mathcal{T}_{Sp}=\bigcup_{(F,D)\in \Lambda _{n,n}}\mathcal{T}_{Sp}(F,D)
\end{equation*}
over $\Lambda _{n,n}$. Then as in the case of the general linear groups, $\mathcal{T}_{Sp}$ 
gives rise to a $\mathbb{C}$-basis for the Hibi algebra for $({Sp}_{2m}, {Sp}_{2n})$.
As in \S \ref{Hibi_mkn}, we shall identify monomials in the Hibi algebra $\mathcal{H}_{Sp}$ 
with tableaux whose columns are elements of $\mathcal{L}_{Sp}$.

\begin{proposition}\label{Sp_Hibi} 
\begin{enumerate}
\item The Hibi algebra $\mathcal{H}_{Sp}$ for $(Sp_{2m}, Sp_{2n})$
is graded by $\Lambda _{n,n}$,  and for each $(F,D) \in \Lambda _{n,n}$, 
$\mathcal{T}_{Sp}(F,D)$ forms a $\mathbb{C}$-basis for the graded component 
$\mathcal{H}_{Sp}(F,D)$ of $\mathcal{H}_{Sp}$.

\item The number of standard tableaux for $(Sp_{2m}, Sp_{2n})$ of shape $F/D$
is equal to the branching multiplicity $m(\tau^D_{2n}, \tau^F_{2m})$
of $\tau^D_{2n}$ in  $\tau^F_{2m}$.
\end{enumerate}
\begin{proof}
From the isomorphism $\mathcal{L}_{Sp} \cong \mathcal{L}^{2n}_{2m,n}$, 
we can easily see that there is a bijection 
between $\mathcal{T}_{Sp}(F,D)$ and $\mathcal{T}^{2n}_{2m} (F,D)$.
Then (1) follows from Lemma \ref{Hibi-basis} and
(2) follows from Proposition \ref{Sp_counting}
\end{proof}
\end{proposition}

\subsection{}\label{order increasing SP}

We remark that every standard tableau for $(Sp_{2m}, Sp_{2n})$ of shape $F/D$
can be realized as a skew semistandard tableau of shape $F/D$ having
entries from $\{u_{n+1}, v_{n+1}, \cdots , u_m, v_m \}$.
For example, for $m=10$ and $n=6$, the standard tableau
of shape $F=(6,5,3,0,0)$ and $D=(4,3,1)$ 
$$[u_1,u_2, u_3] \preceq [u_1,u_2,v_4] \preceq [u_1,u_2,v_4] 
\preceq [u_1,u_4] \preceq [v_4,u_5] \preceq [u_5]$$
in $\mathcal{L}_{Sp} = \mathcal{L}\langle 3, 10 \rangle_3$ can be
identified with the skew semistandard tableau
\begin{equation*}
\begin{Young}
 \  & \   & \    & \     & $v_4$ & $u_5$ \cr
 \  & \   & \   & $u_4$ & $u_5$ \cr
 \  & $v_4$ & $v_4$  \cr
\end{Young}
\end{equation*}
where the empty boxes in $h$-th row are considered as the ones with $u_h$ 
for $1 \leq h \leq n$.

\smallskip

We also remark that, as it is shown in Proposition \ref{joinirred2}, we can attach
an order increasing subset $A_{I}$ of $\Gamma _{2m,n}^{2n}$ to each $I\in 
\mathcal{L}_{Sp}$:
\begin{equation}
A_{I} = \bigcup_{2n \leq j \leq 2m} A_I ^{(j)}
\end{equation}
where $ A_I ^{(j)} \subset \Gamma_{2m,n}^{2n}$ are defined as
\begin{eqnarray*} 
A_{I}^{(2i-1)} &=& \left\{ x_{1}^{(2i-1)}, x_{2}^{(2i-1)}, \cdots, x_{s_{i}}^{(2i-1)} \right\}, \\
A_{I}^{(2i)} &=& \left\{ x_{1}^{(2i)}, x_{2}^{(2i)}, \cdots, x_{t_{i}}^{(2i)} \right\}.
\end{eqnarray*}
Here $s_i$ and $t_i$ are the numbers of elements
in $I$ less than or equal to $u_i$ and $v_i$ respectively.
Then we can relate every element of $\mathcal{T}_{Sp}$ to a sum of 
characteristic functions over these order increasing subsets as given in
Proposition \ref{bijection YT-GT} and (\ref{procedure2}).
This gives a direct proof for Corollary \ref{Sp_Hibi_iso}.

\subsection{}

Now we want to lift the elements of the Hibi algebra $\mathcal{H}_{Sp}$
to construct the stable range branching algebra for $({Sp}_{2m}, {Sp}_{2n})$. 
For this purpose, we briefly review the polynomial model of ${Sp}_{2m}$-representation 
spaces studied in \cite{Ki08}.

\smallskip

From (\ref{grading1}), as a ${GL}_{2m}$ module, $\mathbb{C}[\mathsf{M}_{2m,m}]^{U_{m}}$ 
decomposes into irreducible representations $\rho_{2m}^{F} $ for $\ell (F)\leq m$. 
By taking ${Sp}_{2m}$ as a subgroup of ${GL}_{2m}$, ${Sp}_{2m}\times {GL}_{m}$ is 
acting on the space $\mathsf{M}_{2m,m} \cong \mathbb{C}^{2m} \otimes \mathbb{C}^m$ as 
the action of ${GL}_{2m}\times {GL}_{m}$ given in (\ref{doubleaction}).
 
Then we take the quotient of $\mathbb{C}[\mathsf{M}_{2m,m}]^{U_{m}}$ by
the ideal $\mathcal{I}_{{Sp}}=\sum_{F}\mathcal{I}^{F}$ where 
$\mathcal{I}^{F}$ is the complement space to $\tau _{2m}^{F}$ in 
$\rho _{2m}^{F}$, i.e., $\rho _{2m}^{F}=\tau _{2m}^{F}\oplus \mathcal{I}^{F}$ 
for each $F$ (cf. \cite[\S 17.3]{FH91}). Then this quotient algebra can be taken 
as a polynomial model of 
the flag algebra for ${Sp}_{2m}$ in that it contains exactly one copy of every 
irreducible representation $\tau _{2m}^{F}$:
\begin{eqnarray*}
\mathcal{F}_{{Sp}} &=&\mathbb{C}[\mathsf{M}_{2m,m}]^{U_{m}}/\mathcal{I}_{{Sp}} \\
&=&\sum_{\ell (F)\leq m}\tau _{2m}^{F}
\end{eqnarray*}

Moreover, this decomposition is compatible with the graded structure of the
algebra, i.e., $\tau _{2m}^{F_1}\cdot \tau _{2m}^{F_2}\subset 
\tau_{2m}^{F_1 + F_2}$. Therefore, for the stable range $\ell (F)\leq n$, 
we can consider its subalgebra consisting of $\tau _{2m}^{F}$ with 
$\ell(F)\leq n$:
\begin{equation}
\mathcal{F}_{{Sp}}^{(n)}=\sum_{\ell (F)\leq n}\tau _{2m}^{F}  \label{Sp-flag}
\end{equation}
%

\subsection{}

To describe generators of $\mathcal{F}_{Sp}$, to each 
$I=[w_{1},\cdots ,w_{r}]\in \mathcal{L}\langle {2m} \rangle$ with $r \leq m$, we attach a
determinant function $\delta _{I'}$\ as follows. For 
$Q \in \mathsf{M}_{2m,m}$, we let 
${\delta} _{I'}(Q)$ denote the determinant of the submatrix of $Q=(t_{a,b})$ 
obtained by taking the $i_{1}^{\prime} , i_{2}^{\prime} , \cdots , i_{r}^{\prime} $-th rows 
and the $1,2,\cdots ,r$-th columns:
\begin{equation}\label{SP_determinant}
{\delta} _{I'}(Q)=\det 
\begin{bmatrix}
t_{i_{1}^{\prime} 1} & t_{i_{1}^{\prime}  2} & \cdots & t_{i_{1}^{\prime}  r}
\\ 
t_{i_{2}^{\prime}  1} & t_{i_{2}^{\prime}  2} & \cdots & t_{i_{2}^{\prime} r}
\\ 
\vdots & \vdots & \ddots & \vdots \\ 
t_{i_{r}^{\prime}  1} & t_{i_{r}^{\prime}  2} & \cdots & t_{i_{r}^{\prime}  r}%
\end{bmatrix}
\end{equation}
where $\{ i_{1}^{\prime} , i_{2}^{\prime} , \cdots , i_{r}^{\prime} \}$ is the image of 
the set $\{ w_1, w_2, \cdots, w_r \} \subset \langle 2m \rangle$ under
\begin{eqnarray}\label{f-bijection}
\psi:\{u_1, v_1 ,\cdots, u_m, v_m \} \rightarrow \{1, 2, \cdots, 2m \} \\
\psi(u_c)=c \hbox{ \ and \ } \psi(v_c)=2m+1-c \notag
\end{eqnarray}
for $1 \leq c \leq m$. 

\smallskip

This conversion procedure is to make the labeling $(u_c, v_c)$ of isotropic pairs, 
which are denoted by $(c, \bar{c})$ in \cite{Be86} and $(2c-1, 2c)$ in \cite{Ki08}, 
compatible with ours $(c, 2m+1-c)$ for the skew symmetric form defined in 
\S \ref{sp-section1}.

\begin{notation}\label{Sp-notation}
To avoid a possible ambiguity, we impose  
a new total order $\lessdot$ on $\{1, 2, \cdots, 2m \}$ induced by $\psi$ in (\ref{f-bijection})
and the order of $\langle 2m \rangle$ given in (\ref{Sp-alphabet}):
\begin{equation*}
 1 \lessdot 2m \lessdot 2 \lessdot 2m-1 \lessdot \cdots \lessdot m \lessdot m+1
\end{equation*}
\begin{enumerate}
\item To emphasize the order $\lessdot$, we shall use the prime symbol as in 
$i_{j}^{\prime}$ for the elements $i_j$ of $\{1, 2, \cdots, 2m\}$. 

\item Then, in the determinant (\ref{SP_determinant}), we may further assume that 
$$i_{1}^{\prime} \lessdot i_{2}^{\prime} \lessdot \cdots \lessdot i_{r}^{\prime}$$
to fix the sign of the determinant.

\item We also let $I'$ denote the image of $I \in \mathcal{L}\langle 2m \rangle$ 
under $\psi$. Similarly, we let $\mathsf{t}'$ denote the multiple chain 
$(I'_1 \preceq I'_2 \preceq \cdots \preceq I'_c)$ corresponding to the multiple chain 
$\mathsf{t}=(I_1 \preceq I_2 \preceq \cdots \preceq I_c)$ in $\mathcal{L}\langle 2m \rangle$.
\end{enumerate}
\end{notation}

\medskip

For the flag algebra $\mathcal{F}_{Sp}$, we are interested in ${\delta}_{I'}$ with 
$I \in \mathcal{L}\langle 2m \rangle$ whose $h$-th smallest entry is not less than 
$u_h$ for all $h \geq 0$.

\begin{definition}[{\protect\cite{Be86, Ki08}}]\label{Sp-stm}
Fix the element $J_{0}=[u_1,u_2,\cdots ,u_m]\in \mathcal{L}\langle 2m \rangle$ of
length $m$. For a multiple chain 
$\mathsf{t}=(I_{1}\preceq I_{2}\preceq \cdots \preceq I_{c})$ of 
$\mathcal{L}\langle 2m \rangle$, its associated monomial
\begin{equation*}
{\Delta} _{\mathsf{t}'}={\delta} _{I'_{1}} {\delta} _{I'_{2}}\cdots {\delta} _{I'_{c}}\in 
\mathbb{C}[\mathsf{M}_{2m,m}]^{U_{m}}
\end{equation*}
is called a ${Sp}$-\textit{standard monomial}, if $I_{s}\succeq J_{0}$ for all $s$. 
\end{definition}

\medskip

\subsection{}\label{straightening2}
The ideal $\mathcal{I}_{{Sp}}$ is finitely generated. 
Using the elements of  $\mathcal{I}_{{Sp}}$ (cf. \cite[\S 17.3]{FH91}) combined 
with standard monomial theory 
of $\mathbb{C}[\mathsf{M}_{2m,m}]^{U_m}$,  \cite{Ki08} shows that 
$Sp$-standard monomials project to $\mathbb{C}$-basis elements of the quotient algebra
$\mathcal{F}_{Sp}$, and that they are compatible with the graded structure of the algebra.

\smallskip

To a product of ${\delta}_{I'}$'s, as an element of $\mathbb{C}[\mathsf{M}_{2m,m}]^{U_m}$, apply 
the straightening relations in Proposition \ref{straightening1} to obtain a linear combination of 
standard monomials for $GL_{2m}$:
$$\prod_i {\delta} _{{I'}_i}=\sum_{r} c_{r} \prod_{j \geq 1} {\delta} _{K'_{r,j}}$$
If there is a non-zero term $\prod_j {\delta} _{K'_{r,j}}$ which is not a $Sp$-standard monomial,
then apply relations from the ideal $\mathcal{I}_{{Sp}}$, which replace the entries in $K_{r,j}$'s
corresponding to isotropic pairs $(u_a, v_{a})$ with the sum of entries corresponding to 
$(u_b, v_{b})$ for $a \leq b$, thereby expressing $\prod_{j} {\delta} _{K'_{r,j}}$ 
as a linear combination of $Sp$-standard monomials. For further details, we refer to \cite{Ki08}.  
A combinatorial description of this procedure in the language of tableaux is given in \cite{Be86}.

\begin{proposition}[{\protect\cite[Theorem 5.20]{Ki08}}]\label{Sp_flagSMT}
${Sp}$-standard monomials project to a $\mathbb{C}$-basis
of the flag algebra $\mathcal{F}_{{Sp}}$ for ${Sp}_{2m}$. In particular, for
a Young diagram $F$ with $\ell (F)\leq m$, ${Sp}$-\textit{standard monomials}
of shape $F$ project to a weight basis for the ${Sp}_{2m}$\ irreducible
representation $\tau _{2m}^{F}\subset \mathcal{F}_{{Sp}}$.
\end{proposition}

We also note that, from the graded structure 
$\tau_{2m}^{F_1} \cdot \tau_{2m}^{F_2} \subset \tau_{2m}^{F_1 + F_2}$ of $\mathcal{F}_{Sp}$, 
in order to obtain the subalgebra $\mathcal{F}^{(n)}_{Sp}$ in (\ref{Sp-flag}), it is enough to 
consider ${\delta}_{I'}$'s with $I\in \mathcal{L}\langle 2m \rangle$ and $|I| \leq n$.

\subsection{}\label{Sp-U-invariants} 

From (\ref{general branching algebra}) and (\ref{Sp-flag}), we want to find an explicit 
model for the $U_{{Sp}_{2n}}$-invariant subalgebra of $\mathcal{F}_{{Sp}}^{(n)}$:
\begin{eqnarray*}
\left( \mathcal{F}_{{Sp}}^{(n)}\right) ^{U_{{Sp}_{2n}}} &=&\sum_{\ell (F)\leq
n}\left( \tau _{2m}^{F}\right) ^{U_{{Sp}_{2n}}}  \\
&=&\sum_{\ell (F)\leq n}\sum_{D}m(\tau _{2n}^{D},\tau _{2m}^{F})\left( \tau
_{2n}^{D}\right) ^{U_{{Sp}_{2n}}}  \notag
\end{eqnarray*}

\begin{theorem}\label{B-Sp_standardmonomialtheory}
The subalgebra $\mathcal{B}_{Sp}$ of $\mathcal{F}_{{Sp}}$ generated by 
$$ \left\{{\delta}_{I'}+ \mathcal{I}_{Sp}: I \in \mathcal{L}_{Sp} \right\}$$ is
graded by $\Lambda _{n,n}$; and for each $(F,D)\in \Lambda _{n,n}$ the
$Sp$ standard monomials ${\Delta}_{\mathsf{t}'}$  
corresponding to standard tableaux $\mathsf{t}$ for $({Sp}_{2m},{Sp}_{2n})$
whose shapes are $F/D$ form a 
$\mathbb{C}$-basis of the $(F,D)$-graded component. The dimension of the 
$(F,D)$-graded component is equal to the branching multiplicity of 
$\tau_{2n}^{D}$ in $\tau _{2m}^{F}$.

\begin{proof}
Recall that, for $I \in \mathcal{L}_{Sp} \subset \mathcal{L}\langle 2m \rangle$, we defined  
the polynomial ${\delta}_{I'}$ on the space $\mathsf{M}_{2m,m}$ in (\ref{SP_determinant}). 
By (\ref{L_Sp_elements}) and (\ref{f-bijection}), it is the determinant of a submatrix of 
$Q \in \mathsf{M}_{2m,m}$
obtained by taking consecutive columns $\{1, 2, \cdots, |I| \}$ and either consecutive rows 
$\{1, 2, \cdots, r \}$ or partially consecutive rows $\{1, 2, \cdots, r \} \cup 
\{b_1, \cdots, b_s \}$ or only $\{b_1, \cdots, b_s \}$ of $Q$ for $r \leq n$ and 
$b_i \in \{n+1, n+2, \cdots, 2m-n \}$ for all $i$.

Since the left action of $U_{2n} \subset GL_{2m} $ under the embedding (\ref{Sp  embedding}) 
operates the rows of $\mathsf{M}_{2m,m}$, all the elements ${\delta}_{I'}$ for $I\in \mathcal{L}_{Sp}$ are
invariant under the action of $U_{2n}$ and therefor invariant under the action of $U_{Sp_{2n}}$.
Since the ideal $\mathcal{I}_{Sp}$ is stable under the action of $Sp_{2m}$, the generators of the
algebra $\mathcal{B}_{Sp}$ are invariant under the unipotent subgroup $U_{Sp_{2n}}$ of $Sp_{2n}$, 
and so are their products. Also, since every $I \in \mathcal{L}_{Sp}$ satisfies $|I| \leq n$, 
we have
$
\mathcal{B}_{Sp} \subseteq \left( \mathcal{F}_{{Sp}}^{(n)}\right) ^{U_{{Sp}_{2n}}}.
$

On the other hand, every element in $\mathcal{L}_{Sp}$ is greater than 
$J_{0}=[u_1,u_2,\cdots ,u_m]$ with respect to the tableau order, and therefore standard tableaux 
$\mathsf{t}$ for $({Sp}_{2m},{Sp}_{2n})$ (Definition \ref{BrSp_standardtab} and 
Proposition \ref{Sp_flagSMT}) give rise to ${Sp}$-standard monomials $\Delta_{\mathsf{t}'}$ 
(Definition \ref{Sp-stm}) for $\mathcal{F}_{{Sp}}$. That is, $Sp$-standard monomials 
corresponding to standard tableaux for 
$({Sp}_{2m},{Sp}_{2n})$ project to linearly independent elements in the $U_{Sp_{2n}}$-invariant
subalgebra of $\mathcal{F}_{{Sp}}^{(n)}  \subset \mathcal{F}_{{Sp}}$. They span the whole
$U_{Sp_{2n}}$-invariant subalgebra of $\mathcal{F}_{{Sp}}^{(n)}$, because for each 
$(F,D)\in \Lambda _{n,n}$ the number of standard tableaux in $\mathcal{T}_{Sp}(F,D)$ is equal 
to the multiplicity of $\tau _{2n}^{D}$ in $\tau _{2m}^{F}$ by 
Proposition \ref{Sp_counting}. Furthermore, they are scaled by weight $D$ under 
the action of the diagonal subgroup 
$\{diag(a_{1},\cdots ,a_{n},a_{n}^{-1},\cdots ,a_{1}^{-1})\}$ of ${Sp}_{2n}$ as 
given in (\ref{U-invariant-weight}).
Therefore, standard monomials $\Delta _{\mathsf{t}'}$ with 
$\mathsf{t}\in \mathcal{T}_{Sp}(F,D)$ are the highest weight vectors of the
copies of $\tau _{2n}^{D}$ in $\tau _{2m}^{F}$. This shows that
$\mathcal{B}_{Sp} = \left( \mathcal{F}_{{Sp}}^{(n)}\right) ^{U_{{Sp}_{2n}}}$ and
its graded structure.
\end{proof}
\end{theorem}

In this sense, we call $\mathcal{B}_{Sp}$ the \textit{stable range branching algebra for} 
$({Sp}_{2m}, {Sp}_{2n})$. Recall that we obtained $\mathcal{B}_{Sp}$ by lifting 
the elements of the Hibi algebra $\mathcal{H}_{Sp}$ over the distributive lattice 
$\mathcal{L}_{Sp}$ which is isomorphic to the distributive lattice 
$\mathcal{L}_{2m, n}^{2n}$. Now we compare it with the algebra 
$\mathcal{B}_{2m,n}^{2n}$ (Definition \ref{Bmkn}) obtained from the Hibi algebra 
$\mathcal{H}_{2m,n}^{2n}$ for the general linear groups.

\begin{proposition}
The stable range branching algebra $\mathcal{B}_{{Sp}}$ for 
$({Sp}_{2m},{Sp}_{2n})$ is isomorphic to the length $n$ branching algebra 
$\mathcal{B}_{{2m,n}}^{2n}$ for $(GL_{2m}, GL_{2n})$.

\begin{proof}
From the isomorphism $\mathcal{L}_{Sp} \cong \mathcal{L}_{2m,n}^{2n}$ of distributive 
lattices, with $I \mapsto \hat{I}$, we can consider a bijection between the generating set
of $\mathcal{B}_{Sp}$ and the generating set of  $\mathcal{B}_{2m,n}^{2n}$:
$$\left\{ {\delta}_{I'} + \mathcal{I}_{Sp}: I\in \mathcal{L}_{Sp}\right\} 
\longleftrightarrow  \left\{ \delta_{\hat{I}} : \hat{I} \in \mathcal{L}_{2m,n}^{2n}\right\}$$ 
Then, to see that this bijection gives rise to an algebra isomorphism, let us show
that the straightening relations among $\delta_{\hat{I}}$'s in $\mathcal{B}_{2m,n}^{2n}$
agree with those of $({\delta}_{I'} + \mathcal{I}_{Sp})$'s in $\mathcal{B}_{{Sp}} 
\subset \mathcal{F}_{Sp}$.

As explained in \S \ref{straightening2}, to express a product of ${\delta}_{I'}$'s 
as a linear combination of $Sp$-standard monomials projecting to the quotient 
$\mathcal{F}_{Sp}=\mathbb{C}[\mathsf{M}_{2m,m}]^{U_m} / \mathcal{I}_{Sp}$, 
we first apply the straightening relations in $\mathbb{C}[\mathsf{M}_{2m,m}]^{U_m}$ 
(Proposition \ref{straightening1}) and then relations from 
the ideal $\mathcal{I}_{Sp}$.

For elements $I_i \in \mathcal{L}_{Sp} \subset \mathcal{L}\langle 2m \rangle$, 
the corresponding product $\prod_i {\delta} _{I'_i}$ 
as an element in  $\mathbb{C}[\mathsf{M}_{2m,m}]^{U_m}$ can be 
expressed as a linear combination of $GL_{2m}$ standard monomials:
\begin{equation}\label{straightening3}
\prod_i {{\delta}} _{I'_i}=\sum_{r} c_{r} \prod_{j \geq 1} {{\delta}} _{K'_{r,j}}
\end{equation}
in  $\mathbb{C}[\mathsf{M}_{2m,m}]^{U_m}$. Now we claim that for each non-zero term
$\prod_j {\delta} _{K'_{r,j}}$, its indexes $K_{r,j}$'s form a multiple chain 
in $\mathcal{L}_{Sp}$, i.e., the monomial $\prod_j {\delta} _{K'_{r,j}}$ is already $Sp$-standard, 
and therefore the expression (\ref{straightening3}) provides the $Sp$-standard monomial expression 
of $\prod_i {\delta} _{I'_i}$ projecting to $\mathcal{B}_{{Sp}} \subset \mathcal{F}_{Sp}$.
This follows directly from the quadratic relation  (\ref{straightening}), that is,
for $I,J \in \mathcal{L}_{Sp}$,
$${{\delta}} _{I'} {{\delta}} _{J'}=\sum_{r}c_{r} {{\delta}} _{S'_{r}} {{\delta}} _{T'_{r}}.$$
On the right hand side, for each non-zero term ${\delta}_{S'_r} {\delta}_{T'_r}$, 
the chain $S_r \preceq T_r$ satisfies the condition $S_r \succeq J_0$ and $T_r \succeq J_0$ 
in Definition \ref{Sp-stm}, which can be easily seen from the statement (2) of Corollary 
\ref{straightening123} and the fact that $I$ and $J$ from $\mathcal{L}_{Sp}$ do not contain 
$v_h$ for $1 \leq h \leq n$.

Moreover, by Theorem \ref{B-Sp_standardmonomialtheory} and Proposition \ref{Sp_counting}, 
each $(F,D)$ homogeneous spaces of both algebras are of the same dimension, and
they have $\mathbb{C}$-bases labeled by the same patterns.
Therefore, two graded algebras are isomorphic to each other.
\end{proof}
\end{proposition}

With this characterization $\mathcal{B}_{{Sp}} \cong \mathcal{B}_{2m,n}^{2n}$, 
from Theorem \ref{Deformation}, we have

\begin{corollary}
The stable range branching algebra $\mathcal{B}_{{Sp}}$ for 
$({Sp}_{2m},{Sp}_{2n})$ is\ a flat deformation of the Hibi algebra 
$\mathcal{H}_{Sp}$ for $({Sp}_{2m},{Sp}_{2n})$, which is isomorphic 
to $\mathcal{H}_{2m,n}^{2n}$.
\end{corollary}

\medskip


\section{Stable Range Branching Algebra for $({SO}_{p},{SO}_{q})$}


Through out this section, for $m>n\geq 2$, we set
\begin{eqnarray*}
p &=& 2m+1 \hbox{ or } 2m; \\
q &=& 2n+1 \hbox{ or } 2n; \\
k &=& n    \hbox{\ \ \ \ \ \ \  if } q=2n+1, \\
  &=& n-1  \hbox{\ \  if } q=2n. 
\end{eqnarray*}
Following the same techniques we developed for the
symplectic groups, we construct the stable range branching algebra 
$\mathcal{B}_{{SO}}$ for $({SO}_{p},{SO}_{q})$. The results and their proofs 
in this section are analogous to the case of $({Sp}_{2m},{Sp}_{2n})$.

\subsection{}

Let us review a labeling system for the irreducible rational representations
of ${SO}_{p}$ (cf. \cite[\S 3.1.4]{GW09}). For the even orthogonal group ${O}_{2m}$ 
of rank $m$, every Young diagram $F$ with $\ell (F)<m$ can label exactly one irreducible
representation $\sigma _{2m}^{F}$, which can be also realized as a 
${SO}_{2m}$ irreducible representation. The missing ones for ${SO}_{2m}$ are the
pairs of associated irreducible representations $\sigma _{2m}^{F^{+}}$ and 
$\sigma _{2m}^{F^{-}}$ appearing as the components of irreducible
representations $\sigma _{2m}^{F}$\ of $O_{2m}$ labeled by Young diagrams 
$F$ with $m$ rows, i.e., 
$\sigma _{2m}^{F}=\sigma _{2m}^{F^{+}}\oplus \sigma_{2m}^{F^{-}}$. For the odd 
special orthogonal group ${SO}_{2m+1}$ of rank $m$, every irreducible rational 
representation $\sigma _{2m+1}^{F}$ can be uniquely labeled by a Young diagram 
$F$ with $\ell (F)\leq m$. Then these representations are also ${O}_{2m+1}$-irreducible.

\subsection{}\label{sym-space}

Let $J_{m}=(j_{a,b})$ be the $m \times m$ matrix such that $j_{a,m+1-a}=1$ for 
$1\leq a\leq m$ and $0$ otherwise. Then we define the special orthogonal
groups ${SO}_{2m}$ and ${SO}_{2m+1}$ as the subgroups of ${SL}_{2m}$ and 
${SL}_{2m+1}$ preserving the symmetric bilinear forms on $\mathbb{C}^{2m}$ and 
$\mathbb{C}^{2m+1}$ induced by 
\begin{equation*}
\left[ 
\begin{array}{cc}
0 & J_{m} \\ 
J_{m} & 0%
\end{array}%
\right] \text{ and\ }\left[ 
\begin{array}{ccc}
0 & 0 & J_{m} \\ 
0 & 1 & 0 \\ 
J_{m} & 0 & 0%
\end{array}%
\right]
\end{equation*}
respectively where $0$'s are the zero matrices of proper sizes. Then, the pairs
$(e_{j}, e_{p+1-j})$ of the elementary basis elements for $\mathbb{C}^p$ make 
isotropic pairs with respect to the above symmetric bilinear form. Also, the subgroup of upper 
triangular matrices with $1$'s on the diagonal can be taken as a maximal unipotent 
subgroup of ${SO}_p$. We will denote it by $U_{{SO}_{p}}$.

\smallskip

For $m>n$, let us identify ${SO}_{2n}$ as the subgroup of ${SO}_{p}$
preserving the symmetric bilinear form on the subspace of $\mathbb{C}^{p}$
spanned by $\{e_{j},e_{p+1-j}:1\leq j\leq n\}$. Then we can embed 
${SO}_{2n}$ in ${SO}_{p}$ as follows
\begin{equation*}
\left[ 
\begin{array}{cc}
X & Y \\ 
Z & W%
\end{array}%
\right] \rightarrow \left[ 
\begin{array}{ccc}
X & 0 & Y \\ 
0 & I & 0 \\ 
Z & 0 & W%
\end{array}%
\right]
\end{equation*}
where $X,Y,Z,W$ are blocks of size $n \times n$, $I$ is the $(p-2n) \times (p-2n)$ 
identity matrix, and $0$'s are the zero matrices of proper sizes.
Similarly, we embed ${SO}_{2n+1}$ in ${SO}_{2m+1}$ by
considering the $(2n+1)$-dimensional subspace of $\mathbb{C}^{2m+1}$ spanned
by $\{e_{j},e_{2m+2-j}:1\leq a\leq n\}$ and $e_{m+1}$.
For ${SO}_{2n+1}$ in ${SO}_{2m}$, we use the $(2n+1)$-dimensional subspace
of $\mathbb{C}^{2m}$ spanned by $\{e_{j},e_{2m+1-j}:1\leq j\leq n\}$ and $(e_{m}+e_{m+1})$.

\subsection{}

Our next task is to construct an affine semigroup encoding stable range
branching rules for $({SO}_{p},{SO}_{q})$. Note that $(f_1, \cdots, f_m)  \in \mathbb{Z}^m$ 
is a dominant weight for $SO_{2m+1}$ and $SO_{2m}$, if $f_1 \geq \cdots \geq f_m \geq 0$ 
and $f_1 \geq \cdots \geq f_{m-1} \geq |f_m| \geq 0$ respectively.

\begin{lemma}[{\protect\cite[Theorems 8.1.3 and 8.1.4]{GW09}}]
\label{SO_branching}
\begin{enumerate}
\item  Let $F=(f_1, \cdots, f_m)$ and $D=(d_1, \cdots, d_m)$ be dominant weights for 
$SO_{2m+1}$ and $SO_{2m}$ respectively. Then the branching multiplicity of $\sigma _{2m}^{D}$ 
in $\sigma _{2m+1}^{F}$ is equal to $1$ if $(f_1, \cdots, f_m)$ interlaces 
$(d_1, \cdots, |d_m|)$, i.e.,
\begin{equation*}
\begin{array}{cccccccccc}
f_{1} &  & f_{2} &  & \cdots &  & f_{m-1} & & f_{m} &    \\ 
& d_{1} &  & d_{2} &  & \cdots &  & d_{m-1} & & |d_{m}|   
\end{array}%
\end{equation*}
and $0$ otherwise;

\item  Let $F=(f_1, \cdots, f_m)$ and $D=(d_1, \cdots, d_{m-1})$ be dominant weights for 
$SO_{2m}$ and $SO_{2m-1}$ respectively. Then the branching multiplicity of 
$\sigma _{2m-1}^{D}$ in $\sigma _{2m}^{F}$ is equal to $1$ if $(f_1, \cdots, |f_m|)$ 
interlaces $(d_1, \cdots, d_m)$, i.e.,
\begin{equation*}
\begin{array}{ccccccccc}
f_{1} &  & f_{2} &  & \cdots &  & f_{m-1} & &|f_{m}| \\ 
& d_{1} &  & d_{2} &  & \cdots &  & d_{m-1}   &
\end{array}%
\end{equation*}
and $0$ otherwise.
\end{enumerate}
\end{lemma}

By iterating these results, we may obtain patterns counting the branching multiplicities 
for $(SO_p, SO_q)$. Such patterns are different from the GT patterns for $(GL_p, GL_q)$. 
Within the stable range, however, they are the same as the ones for $(GL_p, GL_q)$ with 
restrictions on lengths.
That is because, as in the case for the symplectic groups, the length restriction 
$\ell(F) \leq k$ forces $\ell(D) \leq k$ via the interlacing conditions in 
Lemma \ref{SO_branching}. Therefore, as is shown in Proposition \ref{Sp_counting} for 
the symplectic groups, we have

\begin{proposition}\label{SO-counting}
Let $F$ and $D$ be Young diagrams with $F\supseteq D$ and $\ell (F)\leq k$.
Then the branching multiplicity $m(\sigma _{q}^{D},\sigma _{p}^{F})$ is
equal to the number of elements in $\mathcal{P}_{p}^{q}(F,D)$, and therefore
it is equal to the number of elements in $\mathcal{T}_{p}^{q}(F,D)$.
\end{proposition}

\smallskip

As in the case of $({GL}_{p},{GL}_{q})$ in (\ref{GT-Patterns}), we can
consider the affine semigroup $\mathcal{P}_{p,k}^{q}$ of the order
preserving maps from the GT poset $\Gamma _{p,k}^{q}$ of length $k$:
\begin{equation*}
\begin{array}{cccccccccc}
x_{1}^{(p)} &  & x_{2}^{(p)} &  & \cdots &  & x_{k}^{(p)} &  &  &  \\ 
& x_{1}^{(p-1)} &  & x_{2}^{(p-1)} &  & \cdots &  & x_{k}^{(p-1)} &  &  \\ 
&  & \ddots &  & \ddots &  & \cdots &  & \ddots &  \\ 
&  &  & x_{1}^{(q)} &  & x_{2}^{(q)} &  & \cdots &  & x_{k}^{(q)}%
\end{array}%
\end{equation*}
to non-negative integers. We call $\mathcal{P}_{p,k}^{q}$ the 
\textit{semigroup for} $({SO}_{p},{SO}_{q})$, and define its associated semigroup
algebra:
\begin{equation*}
\mathbb{C}[\mathcal{P}_{p,k}^{q}]=\bigoplus_{(F,D)\in \Lambda _{k,k}}
\mathbb{C}[\mathcal{P}_{p}^{q}]_{(F,D)} 
\end{equation*}
and call it the \textit{semigroup algebra for} $({SO}_{p},{SO}_{q})$.

\medskip

\subsection{}\label{SO-distlattice}

Let us define the distributive lattice for $({SO}_{p},{SO}_{q})$ and
study its Hibi algebra. We shall closely follow the construction developed in
\S \ref{Sp-distlattice} for the symplectic groups. Consider the ordered letters:
\begin{eqnarray}\label{O-alphabet}
\langle 2m \rangle &=& \left\{ u_1 < v_1 < u_2 < v_2 \cdots < u_m < v_m  \right\}, \\
\langle 2m+1 \rangle &=& \left\{ u_1 < v_1 < u_2 < v_2 \cdots < u_m < v_m  < \infty \right\} \notag
\end{eqnarray}
for $p=2m$ and $2m+1$ respectively.

If we let $\mathcal{L} \langle p \rangle$ denote the set of all non-empty subsets $J$
of $\langle p \rangle$, then on $\mathcal{L} \langle p \rangle$ we can also impose the tableau 
order $\preceq$ as in \S \ref{Sec_taborder} and \S \ref{Sp-distlattice}. Then $\mathcal{L}\langle p \rangle$ is a distributive
lattice isomorphic to $\mathcal{L}_{p}$, as in the case of the symplectic groups, through
the bijection (\ref{iota}) (and $\iota(\infty)=2m+1$ for $p=2m+1$).

\smallskip

Then, we define $\mathcal{L}\langle n, q, p \rangle$ to be the set of nonempty 
subsets $J$ of $\mathcal{L}\langle p \rangle$ of the following forms:
\begin{eqnarray}\label{SO-columns}
&&[u_1, u_2, \cdots , u_c , y_{1},y_{2},\cdots ,y_{s}], \\
&&[u_1, u_2, \cdots , u_c], \notag \\
&&[y_{1}, y_{2}, \cdots , y_{s}] \notag
\end{eqnarray}
where $c \leq n$ and, for $q=2n$ and $2n+1$, 
\begin{eqnarray*}
u_{n+1} \leq y_1 < y_2 < \cdots < y_s; \\
v_{n+1} \leq y_1 < y_2 < \cdots < y_s.
\end{eqnarray*}
respectively. In particular, if $u_c \in J$ for 
$c\leq n$, then $\{ u_h: 1 \leq h \leq c\} \subset J$.

\smallskip

Now, let $\mathcal{L}\langle n,q,  p \rangle_k$ be the subset of 
$\mathcal{L}\langle n, q, p \rangle$ consisting of $J$ with $|J| \leq k$. Then, as 
is the case for the symplectic groups (\S \ref{Sp-distlattice}), we can identify 
$\mathcal{L}\langle n, q, p \rangle_k$ 
with the distributive lattice $\mathcal{L}_{p-q+n, k}^n$, and therefore with $\mathcal{L}_{p,k}^q$
by Corollary \ref{L-conversion}.

\begin{definition}
The distributive lattice for $({SO}_{p},{SO}_{q})$ is $\mathcal{L}\langle n, q, p \rangle_k$, 
and it will be denoted by $\mathcal{L}_{SO}$.
\begin{eqnarray*}
\mathcal{L}_{SO} &=& \mathcal{L}\langle n, q, p \rangle_k \\
                 &\cong&  \mathcal{L}_{p, k}^q
\end{eqnarray*}
\end{definition}

Then we define the \textit{Hibi algebra for} $({SO}_{p},{SO}_{q})$, denoted 
by $\mathcal{H}_{SO}$, to be the Hibi algebra over the distributive lattice 
$\mathcal{L}_{SO}$. From the isomorphism of distributive lattices, we have 
$\mathcal{H}_{SO} \cong \mathcal{H}_{p, k}^q$. Then from 
Proposition \ref{DL-GT} for $(GL_{p}, GL_{q})$, we have

\begin{corollary}\label{SO-Hibi-iso}
There is an algebra isomorphism
\begin{equation*}
\mathcal{H}_{SO} \cong \mathbb{C}[\mathcal{P}_{p,k}^{q}]
\end{equation*}
\end{corollary}

\subsection{}

As in the previous cases (\S \ref{Hibi_mkn}), we shall identify the monomials in 
the Hibi algebra $\mathcal{H}_{SO}$ with tableaux whose columns are elements of 
$\mathcal{L}_{SO}$.

\begin{definition}
A standard tableau $\mathsf{t}$ for $({SO}_{p},{SO}_{q})$ is a multiple
chain $I_{1}\preceq \cdots \preceq I_{s}$ in $\mathcal{L}_{SO}$. 
The shape $sh_{n}(\mathsf{t})$ of $\mathsf{t}$ is $F/D$ where 
$F=(|I_{1}|,\cdots ,|I_{s}|)^{t}$ and $D=(d_{1},\cdots ,d_{n})$
with $d_{r}$ being the number of $u_r$'s in $\mathsf{t}$ for $1\leq r\leq n$.
\end{definition}

We write $\mathcal{T}_{SO}(F,D)$ for the set of all standard
tableaux for $({SO}_{p},{SO}_{q})$ whose shapes are $F/D$, and set
\begin{equation*}
\mathcal{T}_{SO}=\bigcup_{(F,D)\in \Lambda _{k,k}}\mathcal{T}_{SO}(F,D)
\end{equation*}
Then as in the case of the symplectic groups, $\mathcal{T}_{SO}$ 
gives rise to a $\mathbb{C}$-basis for the Hibi algebra for $({SO}_{p},{SO}_{q})$.

\begin{proposition} \label{SO_Hibi}
\begin{enumerate}
\item The Hibi algebra $\mathcal{H}_{SO}$ for $(SO_p, SO_q)$
is graded by $\Lambda _{k,k}$ and $\mathcal{T}_{SO}(F,D)$ forms a 
$\mathbb{C}$-basis of the graded component $\mathcal{H}_{SO}(F,D)$.

\item For $(F,D) \in \Gamma_{k,k}$, the number of standard tableaux 
for $(SO_p, SO_q)$ of shape $F/D$ is equal to the branching multiplicity 
$m(\sigma ^D_q, \sigma ^F_p)$ of $\sigma ^D_q$ in  $\sigma ^F_p$.
\end{enumerate}
\begin{proof}
From the isomorphism $\mathcal{L}_{SO} \cong \mathcal{L}^q_{p,k}$, 
it is straightforward to see that there is a bijection 
between $\mathcal{T}_{SO}(F,D)$ and $\mathcal{T}^q_p (F,D)$.
Then (1) follows from Lemma \ref{Hibi-basis} and
(2) follows from Proposition \ref{SO-counting}.
\end{proof}
\end{proposition}

\subsection{}
We can also find a correspondence between $\mathcal{L}_{SO}$ 
and the set of order increasing subsets of the GT poset $\Gamma_{p,k}^{q}$ in 
the same way explained in \S \ref{order increasing SP}. 
Namely, define the order increasing subset $A_{I}$ of $\Gamma _{p,k}^{q}$ 
corresponding to $I\in \mathcal{L}_{SO}$ as
\begin{equation}
A_{I} = \bigcup_{q \leq j \leq p} 
\left\{ x_{1}^{(j)}, x_{2}^{(j)}, \cdots, x_{s_{j}}^{(j)} \right\}
\end{equation}
where, for $n+1 \leq h \leq m$, $s_{2h-1}$ and $s_{2h}$ are the numbers of elements
in $I$ less than or equal to $u_h$ and $v_h$ respectively; and
$s_{2n}$ is the number of elements in $I$ less 
than $v_n$ and $s_{2m+1}$ is the number of elements in $I$. 
Then every element of $\mathcal{T}_{SO}$ can be related to a sum of 
characteristic functions over these order increasing subsets as given in
Proposition \ref{bijection YT-GT} and (\ref{procedure2}). 
This gives a direct proof for Corollary \ref{SO-Hibi-iso}.

\subsection{}

To construct the stable range branching algebra for $({SO}_{p},{SO}_{q})$, we review the
polynomial model of ${SO}_{p}$-representation spaces studied in \cite{Ki09}.

\smallskip

From (\ref{grading1}), $\mathbb{C}[\mathsf{M}_{p,m}]^{U_{m}}$ consists of 
${GL}_{p}$-irreducible representations $\rho _{p}^{F}$ with $\ell (F)\leq m$.
By taking ${O}_{p}$ as a subgroup of ${GL}_{p}$, ${O}_{p}\times {GL}_{m}$ is
acting on the space $\mathsf{M}_{p,m} \cong \mathbb{C}^p \otimes \mathbb{C}^m$ via 
the action of ${GL}_{p}\times {GL}_{m}$ given in (\ref{doubleaction}). Then we take 
the quotient of $\mathbb{C}[\mathsf{M}_{p,m}]^{U_{m}}$ 
by the ideal $\mathcal{I}_{{O}}=\sum_{F}\mathcal{I}^{F}$ where $\mathcal{I}^{F}$ is 
the complement space to the ${O}_{p}$-irreducible representation $\sigma _{p}^{F}$ 
in $\rho _{p}^{F}$, i.e., $\rho _{p}^{F}=\sigma _{p}^{F}\oplus \mathcal{I}^{F}$ 
for each $F$ (cf. \cite[\S 19.5]{FH91}). 

Then \cite{Ki09} shows that this quotient algebra can be taken as 
a polynomial model for the flag algebra for ${SO}_{p}$ in that it contains exactly 
one copy of each irreducible representation $\sigma _{p}^{F}$ with $\ell (F)\leq m$:
\begin{eqnarray*}
\mathcal{F}_{{SO}} &=& \mathbb{C}[\mathsf{M}_{p,m}]^{U_{m}}/\mathcal{I}_{{O}} \\
                   &=& \sum_{\ell (F)\leq m}\sigma _{p}^{F}
\end{eqnarray*}
and it is graded by Young diagrams, i.e., $\sigma _{p}^{F_1}\cdot \sigma
_{p}^{F_2}\subset \sigma _{p}^{F_1 + F_2}$. We note that $\sigma_{2m}^F$ with $\ell(F)=m$
are irreducible $O_{2m}$ representations, but they are not irreducible as $SO_{2m}$
representations. 

\smallskip

To take the stable range $\ell (F)\leq k$, we consider its subalgebra consisting of 
$\sigma_{p}^{F}$ with $\ell (F)\leq k$:
\begin{equation}
\mathcal{F}_{{SO}}^{(k)}=\sum_{\ell (F)\leq k}\sigma _{p}^{F}  \label{SO-flag}
\end{equation}

\subsection{}

To describe generators of $\mathcal{F}_{SO}$, to each 
$I=[w_{1},\cdots ,w_{r}]\in \mathcal{L}\langle p \rangle$, we attach a
determinant function ${\delta}_{I'}$\ as follows. 

For $Q \in \mathsf{M}_{p,m}$, 
we let ${\delta}_{I'}(Q)$ denote the determinant of the submatrix of 
$Q=(t_{a,b})$ obtained by taking 
the $i_{1}^{\prime} , i_{2}^{\prime} , \cdots , i_{r}^{\prime} $-th rows 
and the $1,2,\cdots ,r$-th columns:
\begin{equation}\label{O_determinant}
{\delta} _{I'}(Q)=\det 
\begin{bmatrix}
t_{i_{1}^{\prime} 1} & t_{i_{1}^{\prime}  2} & \cdots & t_{i_{1}^{\prime}  r}
\\ 
t_{i_{2}^{\prime}  1} & t_{i_{2}^{\prime}  2} & \cdots & t_{i_{2}^{\prime} r}
\\ 
\vdots & \vdots & \ddots & \vdots \\ 
t_{i_{r}^{\prime}  1} & t_{i_{r}^{\prime}  2} & \cdots & t_{i_{r}^{\prime}  r}%
\end{bmatrix}
\end{equation}
where  is $\{ i_{1}^{\prime}, i_{2}^{\prime}, \cdots , i_{r}^{\prime} \}$ is the image of 
the set $\{ w_1, w_2, \cdots, w_r \} \subset \langle p \rangle$ under $\psi_p$:
\begin{eqnarray}\label{SO-psi}
\psi_{2m}:\{u_1, v_1 ,\cdots, u_m, v_m \} \longrightarrow \{1, 2, \cdots, 2m \} \\
\psi_{2m}(u_c)=c \hbox{ \ and \ } \psi_{2m}(v_c)=2m+1-c ; \notag \\ 
\psi_{2m+1}:\{u_1, v_1 ,\cdots, u_m, v_m , \infty \} \longrightarrow \{1, 2, \cdots, 2m , 2m+1\} \notag \\  
\psi_{2m+1}(u_c)=c \hbox{ \ and \ } \psi_{2m+1}(v_c)=2m+2-c  \notag
\end{eqnarray}
for $p=2m$ and $2m+1$ respectively, for $1 \leq c \leq m$ and $\psi_{2m+1}(\infty)=m+1$. 

Then, with the bijection $\psi_p$, we can impose a new order $\lessdot$ on 
$\{1, 2, \cdots, p \}$ induced by the order on $\langle p \rangle$ in (\ref{O-alphabet}):
\begin{eqnarray*}
&& 1 \lessdot 2m \lessdot 2 \lessdot 2m-1 \lessdot \cdots \lessdot m \lessdot m+1; \\
&& 1 \lessdot 2m+1 \lessdot 2 \lessdot 2m \lessdot \cdots \lessdot m \lessdot m+2 \lessdot m+1 
\end{eqnarray*}
and we keep using the convention of $I'$, $\delta_{I'}$ and $\Delta_{\mathsf{t}'}$ 
used for the symplectic groups (Notation \ref{Sp-notation}). This conversion procedure is to make our labeling $(u_c, v_c)$ of isotropic pairs (\S \ref{sym-space}) compatible with those used
in \cite{KW93, Ki09}.

\medskip

To $I=[w_{1},\cdots ,w_{s}]\in \mathcal{L}\langle p \rangle$, we attach a determinant 
function ${\delta}_{I'}$ as we define in (\ref{O_determinant}). For a multiple chain 
$\mathsf{t}=(I_{1}\preceq \cdots \preceq I_{r})$ of $\mathcal{L}\langle p \rangle$, 
let $\mathsf{t}(a,b)$ denote the $a$-th smallest element in the $b$-th column 
$I_{b}$ of the tableau $\mathsf{t}$. Also, let $\alpha _{2c}$ and 
$\beta _{2c}$ be the numbers of elements less than or equal to $v_c$ in $I_1$ and $I_2$ respectively. 

\begin{definition}[{\protect{cf. \cite{KW93, Pr94}}}]\label{O-standard}
Then the corresponding monomial
\begin{equation*}
{\Delta} _{\mathsf{t}'}={\delta} _{I'_{1}} {\delta}_{I'_{2}}\cdots {\delta}_{I'_{r}}\in 
\mathbb{C}[\mathsf{M}_{p,m}]^{U_{m}}
\end{equation*}
is called an $O$-\textit{standard monomial}, if, in the chain $\mathsf{t}=(I_1 \cdots \preceq I_r)$,
\begin{enumerate}
\item $\alpha _{2c}+\beta _{2c}\leq 2c$ for $1\leq c\leq m$, and

\item if $\alpha _{2c}+\beta_{2c}=2c$ for some $c$ with 
$\mathsf{t}(\alpha _{2c},1)=u_c$ and $\mathsf{t}(\beta _{{2c}},b)=v_c$ for some $b$, then 
$\mathsf{t}(\beta_{2c} -1,b)=u_c$.
\end{enumerate}
\end{definition}
In \cite{KW93} and \cite{Pr94}, the above conditions (1) and (2) are 
used to define Young tableaux describing weight basis elements of 
irreducible $O_p$ representations.

\subsection{}\label{straightening_SO}
The ideal $\mathcal{I}_{O}$ is finitely generated. 
Using the elements of  $\mathcal{I}_{{O}}$ (cf. \cite[\S19.5]{FH91}) combined with 
standard monomial theory of
$\mathbb{C}[\mathsf{M}_{p,m}]^{U_m}$,  \cite{Ki09} shows that $O$-standard monomials project 
to $\mathbb{C}$-basis elements of the quotient algebra $\mathcal{F}_{SO}$, and that they are
compatible with the graded structure of the algebra.

\smallskip

To a product of ${\delta}_{I'}$'s in $\mathbb{C}[\mathsf{M}_{p,m}]^{U_m}$, we apply the 
straightening relations in Proposition \ref{straightening1} to obtain a linear combination 
of standard monomials for $GL_{p}$:
$$\prod_i {\delta}_{I'_i}=\sum_{r} c_{r} \prod_{j \geq 1} {\delta}_{K'_{r,j}}$$

If there is a non-zero term  $\prod_{j} \delta_{K'_{r,j}}$ which is not 
an $O$-standard tableau, then apply relations from the ideal
$\mathcal{I}_{{O}}$, which replace the entries of  $K_{r,j}$'s corresponding to isotropic
pairs $(u_a, v_a)$ with the sum of pairs $(u_b, v_b)$'s (and $(\infty, \infty)$ for $p=2m+1$) for 
some $a \leq b$, thereby expressing $\prod_{j} \delta_{K'_{r,j}}$ as a linear
combination of $O$-standard monomials. For further details, we refer to \cite{Ki09}.  
A combinatorial description this straightening procedure in the language of tableaux 
is given in \cite{KW93}.

\smallskip

The following is shown in \cite{Ki09}. See also \cite{KW93} and \cite{Pr94}.

\begin{proposition}[{\protect{\cite[Theorem 3.6,Proposition 3.9]{Ki09}}}]\label{SO_flagSMT}
${O}$-standard monomials project to a $\mathbb{C}$-basis
of the flag algebra $\mathcal{F}_{{SO}}$ for ${SO}_{p}$. In particular, for
a Young diagram $F$ with $\ell (F)\leq m$, ${O}$-\textit{standard monomials}
of shape $F$ form a weight basis for the ${O}_{p}$-irreducible
representation $\sigma_{p}^{F}\subset \mathcal{F}_{{SO}}$.
\end{proposition}

\subsection{}\label{SO-U-invariants} 
Our next task is, from the discussions (\ref{general branching algebra})
and (\ref{SO-flag}), to find an explicit model for the 
$U_{{SO}_{q}}$-invariant subalgebra of $\mathcal{F}_{{SO}}^{(k)}$:
\begin{eqnarray*}
\left( \mathcal{F}_{{SO}}^{(k)}\right) ^{U_{{SO}_{q}}} &=&\sum_{\ell (F)\leq
k}\left( \tau _{p}^{F}\right) ^{U_{{SO}_{q}}}  \\
&=&\sum_{\ell (F)\leq k}\sum_{D}m(\tau _{q}^{D},\tau _{p}^{F})\left( \tau
_{q}^{D}\right) ^{U_{{SO}_{q}}}  
\end{eqnarray*}

\begin{theorem}\label{B-SO_standardmonomialtheory}
The subalgebra $\mathcal{B}_{SO}$ of $\mathcal{F}_{{SO}}$ generated by 
$$ \left\{{\delta}_{I'}+ \mathcal{I}_{SO}: I \in \mathcal{L}_{SO} \right\}$$ is
graded by $\Lambda _{n,n}$, and for each $(F,D)\in \Lambda _{k,k}$ the
$O$ standard monomials ${\Delta}_{\mathsf{t}'}$  
corresponding to standard tableaux $\mathsf{t}$ for $({SO}_{p},{SO}_{q})$
whose shapes are $F/D$ form a 
$\mathbb{C}$-basis of the $(F,D)$-graded component. The dimension of the 
$(F,D)$-graded component is equal to the branching multiplicity of 
$\tau_{q}^{D}$ in $\tau_{p}^{F}$.
\begin{proof}
For $I \in \mathcal{L}_{SO} \subset \mathcal{L}\langle p \rangle$, we defined 
the polynomial ${\delta}_{I'}$ on the space $\mathsf{M}_{p,m}$ in 
(\ref{O_determinant}). 
By (\ref{SO-columns}) and and (\ref{SO-psi}), it is the determinant of a submatrix of 
$Q \in \mathsf{M}_{2m,m}$ obtained by taking consecutive columns $\{1, 2, \cdots, |I| \}$, 
and either consecutive rows $\{1, 2, \cdots, r \}$ or partially consecutive rows 
$\{1, 2, \cdots, r \} \cup \{b_1, \cdots, b_s \}$ or only $\{b_1, \cdots, b_s \}$ of 
$Q$ for $r \leq n$ and $b_i \in \{n+1, n+2, \cdots, p-n\}$.

Since the left action of $U_{q} \subset GL_{p} $, under the embedding given in 
\S \ref{sym-space}, operates the rows of $\mathsf{M}_{p,m}$, 
all the determinants ${\delta} _{I'}$ for $I\in \mathcal{L}_{SO}$ are
invariant under the action of $U_{q}$, and therefor invariant under the action of 
$U_{SO_{q}}$. Since the ideal $\mathcal{I}_{O}$ is stable under the action of $O_{p}$, 
the generators of the algebra $\mathcal{B}_{SO}$ are invariant under the unipotent 
subgroup $U_{SO_{q}}$ of $SO_{q}$, and so are their products. Also, since every 
$I \in \mathcal{L}_{SO}$ satisfies $|I| \leq k$, we have
$\mathcal{B}_{SO} \subseteq \left( \mathcal{F}_{{SO}}^{(k)}\right) ^{U_{{SO}_{q}}}$.

On the other hand, for every chain $I \preceq J$ in $\mathcal{L}_{SO}$, $\delta_{I'} \delta_{J'}$
satisfies the conditions (1) and (2) in Definition \ref{O-standard}, which can be easily seen 
from the statement (3) of Corollary \ref{straightening123} and the fact that $I$ and $J$ from
$\mathcal{L}_{SO}$ do not contain $v_h$ for $1 \leq h \leq n$. This implies that standard monomials
${\Delta}_{\mathsf{t}'}$ corresponding to standard tableaux $\mathsf{t}$ for $({SO}_{p},{SO}_{q})$
project to linearly independent elements in the $U_{SO_{q}}$-invariant subalgebra of
$\mathcal{F}_{{SO}}^{(k)}  \subset \mathcal{F}_{{SO}}$. They span the whole 
$U_{SO_{q}}$-invariant subalgebra of 
$\mathcal{F}_{{SO}}^{(k)}$, because for each $(F,D)\in \Lambda _{k,k}$ the number 
of standard tableaux in $\mathcal{T}_{SO}(F,D)$ is equal to the
multiplicity of $\tau _{q}^{D}$ in $\tau _{p}^{F}$ by 
Proposition \ref{SO-counting}. Furthermore, they are scaled by weight $D$ under 
the action of the diagonal subgroup 
$\{diag(a_{1},\cdots ,a_{n},a_{n}^{-1},\cdots ,a_{1}^{-1})\}$ or 
$\{diag(a_{1},\cdots ,a_{n}, 1 , a_{n}^{-1},\cdots ,a_{1}^{-1})\}$ of ${SO}_{q}$.
Therefore, standard monomials $\Delta_{\mathsf{t}'}$ with 
$\mathsf{t}\in \mathcal{T}_{SO}(F,D)$ are the highest weight vectors of the
copies of $\tau _{q}^{D}$ in $\tau _{p}^{F}$. This shows that
$\mathcal{B}_{SO} = \left( \mathcal{F}_{{SO}}^{(k)}\right) ^{U_{{SO}_{q}}}$
and its graded structure.
\end{proof}
\end{theorem}

In this sense, we call $\mathcal{B}_{SO}$ the \textit{stable range branching algebra for} 
$({SO}_{p}, {SO}_{q})$. Recall that we obtained $\mathcal{B}_{SO}$ by lifting 
the elements of the Hibi algebra $\mathcal{H}_{SO}$ over the distributive lattice 
$\mathcal{L}_{SO}$ which is isomorphic to the distributive lattice 
$\mathcal{L}_{p,k}^{q}$. Now we compare it with the algebra 
$\mathcal{B}_{p,k}^{q}$ (Definition \ref{Bmkn}) obtained from the Hibi algebra 
$\mathcal{H}_{p,k}^{q}$ for the general linear groups.

\begin{proposition}
The stable range branching algebra $\mathcal{B}_{{SO}}$ for 
$({SO}_{p},{SO}_{q})$ is isomorphic to the length $k$ branching algebra $\mathcal{B}_{{p,k}}^{q}$
for $(GL_p, GL_q)$.

\begin{proof}
From the isomorphism $\mathcal{L}_{SO} \cong \mathcal{L}_{p,k}^q$ of distributive 
lattices, with $I \mapsto \hat{I}$, we can consider a bijection between the generating set
of $\mathcal{B}_{SO}$ and the generating set of  $\mathcal{B}_{p,k}^q$:
$$\left\{ {\delta}_{I'} + \mathcal{I}_{O}: I\in \mathcal{L}_{SO}\right\} 
\longleftrightarrow  \left\{ \delta_{\hat{I}} : \hat{I} \in \mathcal{L}_{p,k}^q\right\}$$ 
Then, to see that this bijection gives rise to an algebra isomorphism, let us show
that the straightening relations among $\delta_{\hat{I}}$'s in $\mathcal{B}_{p,k}^q$
agree with those of $(\delta_{I'} + \mathcal{I}_{O})$'s in $\mathcal{B}_{{SO}} 
\subset \mathcal{F}_{SO}$.

As explained in \S \ref{straightening_SO}, to express a product of ${\delta}_{I'}$'s 
as a linear combination of $O$-standard monomials projecting to the quotient 
$\mathcal{F}_{SO}=\mathbb{C}[\mathsf{M}_{p,m}]^{U_m} / \mathcal{I}_{O}$, 
we first apply the straightening relations in $\mathbb{C}[\mathsf{M}_{p,m}]^{U_m}$ 
(Proposition \ref{straightening1}) and then relations from the ideal $\mathcal{I}_{O}$. 

A product of representatives $\prod_i {\delta}_{I'_i}$, as an element in  
$\mathbb{C}[\mathsf{M}_{p,m}]^{U_m}$ can be expressed as a linear combination of $GL_{p}$ 
standard monomials:
\begin{equation}\label{straightening4}
\prod_i {{\delta}} _{I'_i}=\sum_{r} c_{r} \prod_{j\geq 1} {{\delta}}_{K'_{r,j}}
\end{equation}
in  $\mathbb{C}[\mathsf{M}_{p,m}]^{U_m}$. 

Now we claim that for each non-zero term $\prod_{j}{\delta}_{K'_{r,j}}$, 
the indexes $K_{r,j}$'s form a multiple chain in $\mathcal{L}_{SO}$, therefore (\ref{straightening4}) gives $O$-standard monomial expression of $\prod_i {\delta}_{I'_i}$ projecting to $\mathcal{B}_{{SO}} \subset \mathcal{F}_{SO}$. This follows directly from the quadratic relation  (\ref{straightening}). For every chain $I \preceq J$ in $\mathcal{L}_{SO}$, $\delta_{I'} \delta_{J'}$ satisfies the conditions (1) and (2) in Definition \ref{O-standard}, which can be easily seen from the statement (3) of Corollary \ref{straightening123} and the fact that $I$ and $J$ from $\mathcal{L}_{SO}$ do not contain $v_h$ for $1 \leq h \leq n$.

Moreover, from Theorem \ref{B-SO_standardmonomialtheory} and
Proposition \ref{SO-counting}, each $(F,D)$ homogeneous spaces of both algebras 
are of the same dimension with bases labeled by the same patterns.
This shows that two graded algebras are isomorphic to each other.
\end{proof}
\end{proposition}

With this characterization $\mathcal{B}_{{SO}} \cong \mathcal{B}_{p,k}^q$, 
from Theorem \ref{Deformation}, we have

\begin{corollary}
The stable range branching algebra $\mathcal{B}_{{SO}}$ for 
$({SO}_{p},{SO}_{q})$ is\ a flat deformation of the Hibi algebra 
$\mathcal{H}_{SO}$ for $({SO}_{p},{SO}_{q})$, which is isomorphic to $\mathcal{H}_{p,k}^q$.
\end{corollary}

\medskip

\subsection*{Acknowledgment} 
The author thanks Roger Howe and Steven Glenn Jackson 
for insightful conversations regarding several aspects of this work.

\end{document}